\documentclass[11pt,a4paper]{article}
\usepackage{amssymb}
\usepackage{amsmath}
\usepackage{latexsym}
\usepackage{amsthm}

\newtheorem{thm}{Theorem}
\newtheorem{lem}[thm]{Lemma}

\newtheorem{prop}[thm]{Proposition}
\newtheorem*{conj}{Conjecture}
\theoremstyle{definition}
\newtheorem*{remark}{Remark}
\theoremstyle{definition}

\usepackage{xpatch}
\xpatchcmd{\proof}{\itshape}{\normalfont\proofnameformat}{}{}
\newcommand{\proofnameformat}{}

\begin{document}

\renewcommand{\proofnameformat}{\bfseries}

\begin{center}
{\Large\textbf{On the distribution of Sudler products and Birkhoff sums for the irrational rotation}}

\vspace{10mm}

\textbf{Bence Borda}

{\footnotesize Graz University of Technology

Steyrergasse 30, 8010 Graz, Austria

and

Alfr\'ed R\'enyi Institute of Mathematics

Re\'altanoda utca 13--15, 1053 Budapest, Hungary

Email: \texttt{borda@math.tugraz.at}}

\vspace{5mm}

{\footnotesize \textbf{Keywords:} toral translation, Diophantine approximation, quadratic irrational,

Fourier series, central limit theorem, temporal limit theorem}

{\footnotesize \textbf{Mathematics Subject Classification (2020):} 11K60, 37A50, 37E10}
\end{center}

\vspace{5mm}

\begin{abstract}
We study the value distribution of the Sudler product $\prod_{n=1}^N |2 \sin (\pi n \alpha )|$ and the Diophantine product $\prod_{n=1}^N (2e\| n \alpha \|)$ for various irrational $\alpha$, as $N$ ranges in a long interval of integers. At badly approximable irrationals these products exhibit strong concentration around $N^{1/2}$, and at certain quadratic irrationals they even satisfy a central limit theorem. In contrast, at almost every $\alpha$ we observe an interesting anticoncentration phenomenon when the typical and the extreme values are of the same order of magnitude. Our methods are equally suited for the value distribution of Birkhoff sums $\sum_{n=1}^N f(n \alpha )$ for circle rotations. Using Diophantine approximation and Fourier analysis, we find the first and second moment for an arbitrary periodic $f$ of bounded variation, and (almost) prove a conjecture of Bromberg and Ulcigrai on the appropriate scaling factor in a so-called temporal limit theorem. Birkhoff sums also satisfy a central limit theorem at certain quadratic irrationals.
\end{abstract}

\section{Introduction}

The so-called Sudler product
\[ P_N (\alpha ):=\prod_{n=1}^N |2 \sin (\pi n \alpha )| , \qquad \alpha \in \mathbb{R} \]
was first studied by Erd\H{o}s, Szekeres \cite{ESZ} and Sudler \cite{SU} in the context of restricted partition functions. Sudler products later appeared in several areas of mathematics including $q$-series \cite{LU2}, dynamical systems and KAM theory \cite{KT}, Zagier's quantum modular forms \cite{ZA} and hyperbolic knots in algebraic topology \cite{BD}. They also play an important role in a counterexample to the Baker--Gammel--Wills conjecture on Pad\'e approximants \cite{LU1} and the solution of the Ten Martini Problem \cite{AJ}. The ubiquity of the Sudler product is explained by the large number of its different representations. While defined as a trigonometric product, $P_N (\alpha )$ is also the modulus of the $q$-product $(1-q)(1-q^2) \cdots (1-q^N)$ with $q=e^{2 \pi i \alpha}$ on the unit circle. In addition, its logarithm $\log P_N (\alpha ) = \sum_{n=1}^N f(n \alpha )$ is a Birkhoff sum for the irrational rotation with $f(x)=\log |2 \sin (\pi x)|$. For a general overview of Birkhoff sums for toral translations we refer to the survey \cite{DF}. In particular, our results fit into the category of Birkhoff sums with a logarithmic singularity. The main goal of this paper is to study the value distribution of $P_N (\alpha)$ as $N$ ranges over a long interval of integers.

It is not surprising that the asymptotics of $P_N (\alpha)$ as $N \to \infty$ depends sensitively on the Diophantine properties of the irrational $\alpha$. Refining results of Lubinsky \cite{LU2}, in a series of recent papers Aistleitner, Grepstad et al.\ \cite{AB,AB2,ATZ,GKN,GKN2,GNZ} gave precise estimates for the extreme values of $P_N (\alpha)$ in terms of the continued fraction of $\alpha$. Solving a long-standing open problem of Erd\H{o}s and Szekeres, in \cite{GKN} the golden ratio was shown to satisfy
\begin{equation}\label{goldenratio}
1 \ll P_N \bigg( \frac{1+\sqrt{5}}{2} \bigg) \ll N .
\end{equation}
The same holds for some (but not all) quadratic irrationals. Both the upper and the lower bound in \eqref{goldenratio} are sharp; in fact, for an arbitrary irrational $\alpha$ there exist positive constants $C_1(\alpha), C_2(\alpha)$ such that $P_N(\alpha) \le C_1(\alpha)$ and $P_N(\alpha) \ge C_2(\alpha) N$ for infinitely many $N$. The range of oscillations of $P_N (\alpha )$ thus span at least a factor of $N$, with the golden ratio (and certain other irrationals with small partial quotients) minimizing this range. In contrast, our first result shows that $P_N (\alpha)$ concentrates around $N^{1/2}$ for the majority of integers $N$.
\begin{thm}\label{badlyapproxtheorem} Let $\alpha$ be a badly approximable irrational. For any $M \ge 1$ and $t>0$,
\[ \frac{1}{M} \left| \left\{ 1 \le N \le M \, : \, N^{1/2} e^{-t \sqrt{\log 2N}} \le P_N (\alpha ) \le N^{1/2} e^{t \sqrt{\log 2N}} \right\} \right| = 1-O \left( t^{-2} \right) \]
with an implied constant depending only on $\alpha$. In particular, for any sequence $t_N \to \infty$ the set
\[ \left\{ N \in \mathbb{N} \, : \, N^{1/2} e^{-t_N \sqrt{\log N}} \le P_N (\alpha ) \le N^{1/2} e^{t_N \sqrt{\log N}} \right\} \]
has asymptotic density $1$.
\end{thm}
\noindent Recall that an irrational $\alpha$ is called badly approximable if $\inf_{q \ge 1} q \| q \alpha \| >0$, where $\| \cdot \|$ denotes the distance from the nearest integer. What we actually show is that for any badly approximable $\alpha$,
\begin{equation}\label{expectedbadlyapprox}
\frac{1}{M} \sum_{N=1}^M \log P_N (\alpha ) = \frac{1}{2} \log M + O(\log \log M)
\end{equation}
and
\begin{equation}\label{variancebadlyapprox}
\log M \ll \frac{1}{M} \sum_{N=1}^M \left( \log P_N (\alpha ) -\frac{1}{2} \log M \right)^2 \ll \log M
\end{equation}
with implied constants depending only on $\alpha$. From a probabilistic point of view the previous two formulas represent the expected value and the variance of $\log P_N (\alpha)$ when $N$ is chosen randomly from $[1,M]$, and Theorem \ref{badlyapproxtheorem} follows from the Chebyshev inequality. Lubinsky \cite{LU2} proved that for a badly approximable $\alpha$ the extreme values are
\[ \log M \ll \max_{1 \le N \le M} |\log P_N (\alpha )| \ll \log M . \]
Thus the square root of the variance is negligible compared to the extreme values, explaining the higher concentration in Theorem \ref{badlyapproxtheorem} compared to \eqref{goldenratio}.

The estimates for the variance and the extreme values in the previous two formulas are optimal in the sense that both quantities can oscillate between two different positive constants times $\log M$ for a suitable badly approximable $\alpha$. For quadratic irrational $\alpha$, however, we can say more:
\begin{equation}\label{variancequadratic}
\frac{1}{M} \sum_{N=1}^M \left( \log P_N (\alpha ) -\frac{1}{2} \log M \right)^2 = \sigma (\alpha )^2 \log M + O((\log \log M)^4)
\end{equation}
with an explicitly computable constant $\sigma (\alpha ) >0$, and \cite{AB}
\[ \max_{1 \le N \le M} \log P_N (\alpha ) = c(\alpha ) \log M +O(1), \qquad \min_{1 \le N \le M} \log P_N (\alpha ) = (1-c(\alpha )) \log M +O(1) \]
with some constant $c(\alpha ) \ge 1$ and implied constants depending only on $\alpha$. Here e.g.\ $c((1+\sqrt{5})/2)=1$ (cf.\ formula \eqref{goldenratio}), but the precise value of $c(\alpha )$ is not known in general. For certain simple quadratic irrationals including the golden ratio and $\sqrt{2}$, we also prove a central limit theorem with an estimate for the rate of convergence in the Kolmogorov metric.
\begin{thm}\label{CLTtheorem} Let $a \ge 1$ be an integer, and consider the quadratic irrational
\[ \alpha = [0;a,a,a,\ldots] = \frac{\sqrt{a^2+4}-a}{2} . \]
For any $M \ge 3$ and $t \in \mathbb{R}$,
\[ \frac{1}{M} \bigg| \bigg\{ 1 \le N \le M \, : \, \frac{\log P_N (\alpha ) - \frac{1}{2} \log N}{\sigma (\alpha ) \sqrt{\log N}} \le t \bigg\} \bigg| = \int_{-\infty}^{t} \frac{e^{-x^2/2}}{\sqrt{2 \pi}} \, \mathrm{d} x + O \left( \frac{(\log \log M)^2}{(\log M)^{1/6}} \right) \]
with an implied constant depending only on $\alpha$.
\end{thm}

Our approach is to consider $\log P_N (\alpha ) =\sum_{n=1}^N f(n \alpha )$ as a Birkhoff sum for the irrational rotation with $f(x)=\log |2 \sin (\pi x)|$, and to exploit the fact that $f(x)$ has the particularly simple Fourier series expansion
\begin{equation}\label{fourierseries}
\log |2 \sin (\pi x)| = - \sum_{m=1}^{\infty} \frac{\cos (2 \pi m x)}{m}.
\end{equation}
The similar Birkhoff sum
\[ S_N (\alpha ):= \sum_{n=1}^N \left( \{ n \alpha \} -1/2 \right) , \]
where $\{ \cdot \}$ denotes fractional part, was studied in detail by Beck \cite{BE1,BE2,BE3}. Note that $\log |2 \sin (\pi x)|$ and $\pi (\{ x \} -1/2)$ are harmonic conjugates: they are the real and imaginary parts of $\log (1-z)$ (holomorphic on $\mathbb{C} \backslash [1,\infty )$ and defined using the principal branch of the logarithm) on the unit circle $z=e^{2 \pi i x}$. Beck proved that for a quadratic irrational $\alpha$ the expected value and the variance\footnote{In fact, Beck proved the formula for the variance with error term $O(\sqrt{\log M} \log \log M)$. Following the methods of Section \ref{generalsection} of the present paper, this can be improved to $O((\log \log M)^4)$.} of $S_N (\alpha)$ are
\[ \begin{split} \frac{1}{M} \sum_{N=1}^M S_N (\alpha ) &= E(\alpha ) \log M +O(1), \\ \frac{1}{M} \sum_{N=1}^M \left( S_N (\alpha ) - E (\alpha ) \log M \right)^2 &= \frac{\sigma (\alpha )^2}{\pi^2} \log M + O \left( (\log \log M)^4 \right) , \end{split} \]
and $S_N (\alpha )$ satisfies the central limit theorem.
\begin{thm}[Beck \cite{BE1,BE2,BE3}]\label{Becktheorem} Let $\alpha$ be a quadratic irrational. For any $M \ge 3$ and $t \in \mathbb{R}$,
\[ \frac{1}{M} \bigg| \bigg\{ 1 \le N \le M \, : \, \frac{S_N (\alpha ) - E(\alpha ) \log N}{(\sigma (\alpha ) /\pi ) \sqrt{\log N}} \le t \bigg\} \bigg| = \int_{-\infty}^{t} \frac{e^{-x^2/2}}{\sqrt{2 \pi}} \, \mathrm{d} x + O \left( \frac{\log \log M}{(\log M)^{1/10}} \right) \]
with an implied constant depending only on $\alpha$.
\end{thm}
\noindent Here $E(\alpha )$ is defined the following way: given the continued fraction expansion $\alpha =[a_0;a_1,a_2,\dots ]$ with convergents $p_k/q_k = [a_0;a_1, \dots , a_k]$, we let
\[ E(\alpha ) = \lim_{k \to \infty} \frac{1}{12 \log q_k} \sum_{\ell =1}^k (-1)^{\ell} a_{\ell} . \]
Note that the limit exists by the periodicity of the partial quotients. For instance, $E((1+\sqrt{5})/2)=0$ (since $(1+\sqrt{5})/2=[1;1,1,\dots]$ so there is perfect cancellation in the alternating sum), but $E(\sqrt{3}) = 1/(12 \log (2+\sqrt{3}))$ (since $\sqrt{3}=[1;1,2,1,2,\dots ]$ and $q_k \approx (2+\sqrt{3})^{k/2}$). In contrast, the corresponding constant factor in the expected value of $\log P_N (\alpha )$ is always $1/2$, see \eqref{expectedbadlyapprox}.

The variances of $S_N (\alpha)$ and $\log P_N (\alpha)$, on the other hand, involve the same constant $\sigma (\alpha )$. The extra factor of $\pi^2$ is explained by the harmonic conjugate property mentioned above. The constant $\sigma (\alpha )$ also appears in the variance of certain lattice point counting problems \cite{BE1} and in the $L^2$-discrepancy of irrational lattices \cite{BO}. We refer to Beck \cite[Chapter 3]{BE1} for a way to compute $\sigma (\alpha )$ for general quadratic irrationals, based on deep arithmetic properties of the real quadratic field $\mathbb{Q} (\alpha )$. For instance, we have
\[ \sigma \bigg( \frac{1+\sqrt{5}}{2} \bigg)^2 = \frac{\pi^2}{60 \sqrt{5} \log \frac{1+\sqrt{5}}{2}} \quad \textrm{and} \quad \sigma (\sqrt{3})^2 = \frac{\pi^2}{24 \sqrt{3} \log (2+\sqrt{3})} . \]

Using the harmonic conjugate property, Theorems \ref{CLTtheorem} and \ref{Becktheorem} can actually be combined into a 2-dimensional central limit theorem, where the limit distribution is a mean zero Gaussian with the $2 \times 2$ identity matrix as covariance matrix.
\begin{thm}\label{2dimCLTtheorem} Let $a \ge 1$ be an integer, and consider the quadratic irrational
\[ \alpha = [0;a,a,a,\ldots] = \frac{\sqrt{a^2+4}-a}{2} . \]
For any $M \ge 3$ and any convex set $C \subseteq \mathbb{R}^2$,
\begin{multline*}
\frac{1}{M} \bigg| \bigg\{ 1 \le N \le M \, : \, \left( \frac{\log P_N (\alpha ) - \frac{1}{2} \log N}{\sigma (\alpha) \sqrt{\log N}}, \frac{S_N(\alpha ) -E(\alpha ) \log N}{(\sigma (\alpha ) /\pi ) \sqrt{\log N}} \right) \in C \bigg\} \bigg| \\ = \int_{C} \frac{e^{-|x|^2/2}}{2 \pi} \, \mathrm{d} x + O \left( \frac{(\log \log M)^2}{(\log M)^{1/6}} \right)
\end{multline*}
with an implied constant depending only on $\alpha$.
\end{thm}

Our proof of Theorems \ref{CLTtheorem} and \ref{2dimCLTtheorem} follows closely the rather long proof of Theorem \ref{Becktheorem} due to Beck. We decided to focus on the special irrationals $\alpha=[0;a,a,a,\ldots]$ because for these some of the technical complications can be avoided, resulting in a proof of reasonable length. Nevertheless, we believe that Beck's methods are fully applicable to Sudler products, and in particular that Theorems \ref{CLTtheorem} and \ref{2dimCLTtheorem} hold for all quadratic irrationals. Note that we kept the centering term $E(\alpha) \log N$ in Theorem \ref{2dimCLTtheorem} to emphasize this generality, even though for $\alpha=[0;a,a,a,\ldots]$ we actually have $E(\alpha)=0$. We conjecture that the sharp rate of convergence in all three theorems is $O((\log M)^{-1/2})$.

Beck \cite[p.\ 247]{BE1} also announced a more general form of Theorem \ref{Becktheorem}, stating that the central limit theorem remains true (with suitable centering and scaling, and without an explicit error term) provided that the continued fraction $\alpha =[a_0;a_1,a_2,\dots ]$ satisfies $a_k/(a_1^2+\cdots +a_k^2)^{1/2} \to 0$ as $k \to \infty$; he calls this an analogue of the classical Lindeberg condition for the central limit theorem in probability theory. A detailed proof has not been published yet. For this reason we conjecture that our Theorems \ref{CLTtheorem} and \ref{2dimCLTtheorem} in fact hold (with suitable centering and scaling, and without an explicit error term) under the same condition.

The distribution of the Birkhoff sum
\[ S_N (\alpha ,f) := \sum_{n=1}^N f(n \alpha ) \]
has also been studied with more general $1$-periodic functions $f$. Beck \cite{BE1,BE2,BE3} considered a quadratic irrational $\alpha$ and $f(x)=I_{[a,b]}(\{ x \}) -(b-a)$, where $I_{[a,b]}$ is the indicator of an interval $[a,b] \subset [0,1]$ with rational endpoints, and proved a central limit theorem analogous to Theorem \ref{Becktheorem} with $E(\alpha)$ and $\sigma (\alpha)$ replaced by explicitly computable constants depending on $\alpha$ and the interval $[a,b]$. More recently Dolgopyat et al.\ \cite{ADDS,BU,DF,DS1,DS2,DS3} proved similar temporal limit theorems for the irrational rotation as well as Anosov flows and horocycle flows. Here ``temporal'' refers to the fact that the time parameter $N$ is chosen randomly. The centering and scaling terms, however, were not explicitly computed in some of these results. In particular, generalizing Beck's result Bromberg and Ulcigrai \cite{BU} considered a badly approximable $\alpha$ and $f(x)=I_{[a,b]}(\{ x \}) -(b-a)$ with an interval $[a,b] \subset [0,1]$ whose length $0<b-a<1$ is badly approximable with respect to $\alpha$ in the sense that $\inf_{q \in \mathbb{Z} \backslash \{ 0 \}} |q| \cdot \| q \alpha -(b-a) \| >0$. They proved that $(S_N(\alpha ,f)-A_M)/B_M$ converges to the standard Gaussian distribution with the natural centering and scaling terms
\begin{equation}\label{AMBM}
A_M=\frac{1}{M} \sum_{N=1}^M S_N (\alpha, f) \quad \textrm{and} \quad B_M^2=\frac{1}{M}\sum_{N=1}^M (S_N(\alpha, f)-A_M)^2 ,
\end{equation}
and conjectured that $\log M \ll B_M^2 \ll \log M$. In this paper we basically settle their conjecture.
\begin{thm}\label{BUtheorem} Let $\alpha$ be a badly approximable irrational, and let $f(x)=I_{[a,b]}(\{ x \})-(b-a)$ with an interval $[a,b] \subset [0,1]$ of length $0<b-a<1$ such that $\inf_{q \in \mathbb{Z} \backslash \{ 0 \}} |q| \cdot \| q \alpha -$ $k(b-a) \| >0$ for all $k \in \{ 1,3,5,7,9 \}$. Then
\[ \log M \ll B_M^2 \ll \log M \]
with implied constants depending only on $\alpha$ and $(b-a)$.
\end{thm}
\noindent Note that some condition on the length $(b-a)$ is necessary. Indeed, by a classical result on bounded remainder intervals \cite{KE}, we have $|S_N(\alpha, f)| \ll 1$ if and only if $b-a \in \mathbb{Z} \alpha + \mathbb{Z}$; in this case $B_M \ll 1$.

In fact, we prove much more: we find the expected value and the variance of $S_N(\alpha, f)$ for all $1$-periodic functions of bounded variation, see Proposition \ref{momentsprop2}. As we will see, the upper bound $B_M^2 \ll \log M$ remains true for all $f$ and badly approximable $\alpha$. In contrast, the lower bound for $B_M^2$ depends on a delicate interplay between the Fourier coefficients of $f$ and the Diophantine properties of $\alpha$. For certain simple quadratic irrationals we also prove a central limit theorem under the assumption that $B_M \to \infty$ fast enough.
\begin{thm}\label{BVCLTtheorem} Let $a \ge 1$ be an integer, and consider the quadratic irrational
\[ \alpha = [0;a,a,a,\ldots] = \frac{\sqrt{a^2+4}-a}{2} . \]
Let $f: \mathbb{R} \to \mathbb{R}$ be a $1$-periodic function which is of bounded variation on $[0,1]$ and satisfies $\int_0^1 f(x) \, \mathrm{d}x=0$, and let $V(f)$ denote its total variation on $[0,1]$. Let $A_M$, $B_M$ be as in \eqref{AMBM}. For any $M \ge 3$ and $t \in \mathbb{R}$,
\[ \frac{1}{M} \left| \left\{ 1 \le N \le M \, : \, \frac{S_N (\alpha, f)-A_M}{B_M} \le t \right\} \right| = \int_{-\infty}^{t} \frac{e^{-x^2/2}}{\sqrt{2 \pi}} \, \mathrm{d} x + O \left( \frac{V(f) (\log M)^{1/3} (\log \log M)^2}{B_M} \right) \]
with an implied constant depending only on $\alpha$.
\end{thm}
\noindent We conjecture that the same holds for all quadratic irrationals, perhaps for even more general $\alpha$, and that the sharp rate of convergence is $O(V(f)/B_M)$. In particular, we expect a central limit theorem to hold as soon as $B_M \to \infty$.

It would also be interesting to estimate higher moments of $\log P_N (\alpha)$ and $S_N (\alpha ,f)$. Such results could improve the error term $O(t^{-2})$ in Theorem \ref{badlyapproxtheorem}, and possibly the rate of convergence in Theorems \ref{CLTtheorem}, \ref{Becktheorem}, \ref{2dimCLTtheorem} and \ref{BVCLTtheorem}.

In \cite{AB,AB2} we studied the asymptotics of $\sum_{N=1}^M P_N (\alpha )^t$ with $t>0$. The original motivation came from the fact that at $t=2$ this sum is related to quantum modular forms in number theory \cite{ZA} and to quantum invariants of hyperbolic knots in algebraic topology \cite{BD}. From a probabilistic point of view the same sum corresponds to the moment-generating function of $\log P_N (\alpha )$ at the point $t$, which is closely related to temporal large deviations. It would be interesting to see if the methods recently developed for the Sudler product can be adapted to more general Birkhoff sums $S_N (\alpha ,f)$.

The rest of the paper is organized as follows. In Sections \ref{esection} and \ref{almosteverysection} we discuss non-badly approximable irrationals such as Euler's number $e$ and almost every $\alpha$. In Section \ref{diophantineproductsection} we consider the Diophantine product $\prod_{n=1}^N \| n \alpha \|$, and show that its behavior can be reduced to that of $P_N(\alpha)$. In Section \ref{sudlerdistribution} we prove our results on Sudler products: Theorem \ref{badlyapproxtheorem} and formulas \eqref{expectedbadlyapprox}--\eqref{variancequadratic} are proved in Section \ref{badlysection}, Theorems \ref{CLTtheorem} and \ref{2dimCLTtheorem} are proved in Section \ref{CLTsection}, the results on Euler's number $e$ are proved in Section \ref{nonbadlysection}, and the results on almost every $\alpha$ are proved in Sections \ref{aeasymptoticssection} and \ref{measuresection}. In Section \ref{birkhoffdistribution} we prove results on Birkhoff sums with functions of bounded variation, including Theorems \ref{BUtheorem} and \ref{BVCLTtheorem}.

\section{Further results}

\subsection{Euler's number}\label{esection}

Our methods generalize to other types of irrationals. As an illustration, consider Euler's number $e$. It has been known since Euler himself that the continued fraction of $e$ is
\[ e=[2;1,2,1,1,4,1, \dots, 1,2n,1,\dots ] . \]
In particular, $e$ is only logarithmic factors away from being badly approximable in the sense that $\inf_{q \ge 3} \frac{q \log q}{\log \log q} \| q e\| >0$. We have a concentration result similar to the case of badly approximable irrationals.
\begin{thm}\label{etheorem} For any $M \ge 1$ and $t>0$,
\[ \frac{1}{M} \left| \left\{ 1 \le N \le M \, : \, \exp \bigg( - t \left( \frac{\log 3N}{\log \log 3N} \right)^{\frac{3}{2}} \bigg) \le P_N (e) \le \exp \left( t \left( \frac{\log 3N}{\log \log 3N} \right)^{\frac{3}{2}} \right) \right\} \right| = 1-O \left( t^{-2} \right) \]
with a universal implied constant. In particular, for any sequence $t_N \to \infty$ the set
\[ \bigg\{ N \in \mathbb{N} \, : \, \exp \bigg( - t_N \left( \frac{\log N}{\log \log N} \right)^{\frac{3}{2}} \bigg) \le P_N (e) \le \exp \bigg( t_N \left( \frac{\log N}{\log \log N} \right)^{\frac{3}{2}} \bigg) \bigg\} \]
has asymptotic density $1$.
\end{thm}
\noindent What we actually prove is that
\begin{equation}\label{eexpected}
\left| \frac{1}{M} \sum_{N=1}^M \log P_N (e) \right| \ll \log M
\end{equation}
and
\begin{equation}\label{evariance}
\frac{1}{M} \sum_{N=1}^M (\log P_N (e))^2 = \frac{\pi^2}{540} \left( \frac{\log M}{\log \log M} \right)^3 \bigg( 1+O \left( \frac{\log \log \log M}{\log \log M} \right) \bigg) .
\end{equation}
Observe that now the expected value is negligible compared to the square root of the variance, which explains why $P_N (e)$ oscillates around $1$ instead of $N^{1/2}$. Since the continued fraction satisfies the Lindeberg-type condition $a_k/(a_1^2+\cdots +a_k^2)^{1/2} \to 0$ of Beck, we conjecture that $\log P_N (e)$ satisfies the central limit theorem.
\begin{conj} For any $t \in \mathbb{R}$,
\[ \frac{1}{M} \Big| \Big\{ 1 \le N \le M \, : \, \frac{\log P_N (e)}{\frac{\pi}{\sqrt{540}} \left( \frac{\log N}{\log \log N} \right)^{3/2}} \le t \Big\} \Big| \to \int_{-\infty}^{t} \frac{e^{-x^2/2}}{\sqrt{2 \pi}} \, \mathrm{d}x \quad \textrm{as } M \to \infty . \]
\end{conj}
\noindent From general results in \cite{AB} it follows that once again the extreme values are of larger order of magnitude than the square root of the variance:
\begin{equation}\label{eextremal}
\begin{split} \max_{1 \le N \le M} \log P_N (e) &= V \left( \frac{\log M}{\log \log M} \right)^2 \bigg( 1+O \left( \frac{\log \log \log M}{\log \log M} \right) \bigg) , \\ \min_{1 \le N \le M} \log P_N (e) &= -V \left( \frac{\log M}{\log \log M} \right)^2 \bigg( 1+O \left( \frac{\log \log \log M}{\log \log M} \right) \bigg) \end{split}
\end{equation}
with
\begin{equation}\label{V}
V= \int_0^{5/6} \log |2 \sin (\pi x)| \, \mathrm{d}x \approx 0.1615.
\end{equation}
The constant $V$ originally appeared in asymptotic results on Sudler products in \cite{AB,BD} via a connection to algebraic topology: $4 \pi V$ is the hyperbolic volume of the complement of the figure-eight knot. A more elementary, arithmetic interpretation of $V$ was given in \cite{AB2}.

Euler's number is of course just an illustration, and the results of this section apply (with suitable constant factors) to any irrational whose partial quotients are given by a linear pseudopolynomial; that is, there exist finitely many polynomials $p_1 (x), \dots, p_m (x)$ of degree at most $1$, with at least one of them of degree equal to $1$, and a positive integer $k_0$ such that $a_k=p_{k \pmod{m}}(k)$ for all $k \ge k_0$. Such irrationals include an infinite family of numbers related to $e$.

\subsection{Almost every irrational}\label{almosteverysection}

It is an unfortunate fact that the only ``concrete'' irrationals whose continued fraction expansion is explicitly known are the quadratic irrationals and the infinite family of irrationals related to $e$ mentioned in the previous section. Very little is known about the continued fraction of other classical irrationals such as $\pi$, $\sqrt[3]{2}$ or $\log 2$, although numerical experiments suggest they exhibit random behavior. We do know, however, the precise statistics of the partial quotients $a_k$ of almost every real $\alpha$ in the sense of the Lebesgue measure.

The extreme values of $\log P_N (\alpha )$ for a.e.\ $\alpha$ are as follows. Let $\varphi (x)$ be a nondecreasing, positive real-valued function on $(0,\infty )$. If $\sum_{k=1}^{\infty} 1/\varphi (k) < \infty$, then for a.e.\ $\alpha$,
\begin{equation}\label{aeupperbound}
|\log P_N (\alpha ) | \le \varphi (\log N) + O (\log N \log \log N)
\end{equation}
with an implied constant depending only on $\alpha$ and $\varphi$. If $\sum_{k=1}^{\infty} 1/\varphi (k) = \infty$, then for a.e.\ $\alpha$,
\begin{equation}\label{aelowerbound}
\log P_N(\alpha ) \ge \varphi (\log N) \quad \textrm{and} \quad \log P_N(\alpha ) \le - \varphi (\log N) \quad \textrm{for infinitely many } N.
\end{equation}
Thus e.g.\ for any $\varepsilon >0$, for a.e.\ $\alpha$ there exists a constant $C=C(\alpha, \varepsilon)>0$ such that $|\log P_N (\alpha )| \le C \log N (\log \log N)^{1+\varepsilon}$, but this fails with $\varepsilon =0$. Such a convergence/divergence criterion was first proved by Lubinsky \cite{LU2}, who in fact explicitly constructed an exponentially increasing sequence of $N$'s (each having a single nonzero digit in its Ostrowski expansion) satisfying the second inequality in \eqref{aelowerbound}. For the sake of completeness, we explain in Section \ref{nonbadlysection} how to derive \eqref{aeupperbound} and \eqref{aelowerbound} from general results in \cite{AB,BD}. In this paper we improve the lower bound \eqref{aelowerbound}: we show that both inequalities hold on a set of positive upper asymptotic density.
\begin{thm}\label{aeupperdensitytheorem} Let $\varphi (x)$ be a nondecreasing, positive real-valued function on $(0,\infty )$. If $\sum_{k=1}^{\infty} 1/\varphi (k) = \infty$, then for a.e.\ $\alpha$ the sets
\[ \left\{ N \in \mathbb{N} \, : \, \log P_N (\alpha ) \ge \varphi (\log N) \right\} \]
and
\[ \left\{ N \in \mathbb{N} \, : \, \log P_N (\alpha ) \le - \varphi (\log N) \right\} \]
have upper asymptotic density at least $\pi^2/(1440 V^2) \approx 0.2627$, where $V$ is as in \eqref{V}.
\end{thm}
\noindent What we actually prove is that for a.e.\ $\alpha$ there exists a sequence $\psi_{\alpha}(M) \to 0$ such that
\begin{equation}\label{aevariance}
\sqrt{\frac{1}{M} \sum_{N=1}^M (\log P_N (\alpha ))^2} = \left( \frac{\pi}{\sqrt{720}V} +\psi_{\alpha}(M) \right) \max_{1 \le N \le M} |\log P_N (\alpha )| \quad \textrm{for infinitely many } M.
\end{equation}
This means that the typical and the extreme oscillations of $\log P_N (\alpha)$ are of the same order of magnitude; we call it an anticoncentration result. We believe that the value $\pi/(\sqrt{720}V)$ is sharp, i.e.\ for a.e.\ $\alpha$,
\[ \limsup_{M \to \infty} \sqrt{\frac{1}{M} \sum_{N=1}^M (\log P_N (\alpha ))^2} \left( \max_{1 \le N \le M} |\log P_N (\alpha )| \right)^{-1} = \frac{\pi}{\sqrt{720}V} , \]
but this remains open. On the other hand, $\pi^2/(1440 V^2)$ in Theorem \ref{aeupperdensitytheorem} is likely not optimal; perhaps $\{ N \in \mathbb{N} \, : \, |\log P_N (\alpha )| \ge \varphi (\log N) \}$ even has upper asymptotic density $1$. Note that \eqref{aevariance} immediately implies that $\log P_N (\alpha )$ does not satisfy the central limit theorem (with any centering and scaling) for a.e.\ $\alpha$. See Dolgopyat and Sarig \cite{DS1,DS2} for why the central limit theorem fails also for $S_N (\alpha ,f )$ with certain functions of bounded variation for a.e.\ $\alpha$.

The convergence/divergence criterion \eqref{aeupperbound}, \eqref{aelowerbound} can be restated as follows. Under a suitable regularity condition on $\varphi (x)$ in the case of convergence (e.g.\ $\varphi (x) /(x \log x) \to \infty$), for a.e.\ $\alpha$,
\[ \limsup_{M \to \infty} \frac{\displaystyle \max_{1 \le N \le M} \log P_N (\alpha)}{\varphi (\log M)} = \left\{ \begin{array}{ll} 0 & \textrm{if } \sum_{k=1}^{\infty} 1/\varphi (k) < \infty , \\ \infty & \textrm{if } \sum_{k=1}^{\infty} 1/\varphi (k) = \infty , \end{array} \right. \]
and the same holds with $\max$ replaced by $-\min$. Our next result establishes the liminf behavior of the same sequences.
\begin{thm}\label{aeliminftheorem} For a.e.\ $\alpha$,
\[ \liminf_{M \to \infty} \frac{\displaystyle \max_{1 \le N \le M} \log P_N (\alpha)}{\log M \log \log M} = \frac{12 V}{\pi^2} . \]
The same holds with $\max$ replaced by $-\min$.
\end{thm}
\noindent It easily follows that for a.e.\ $\alpha$ the set
\[ \left\{ N \in \mathbb{N} \, : \, |\log P_N (\alpha )| \ge \left( \frac{12 V}{\pi^2} + \varepsilon \right) \log N \log \log N \right\}, \qquad \varepsilon >0 \]
has zero lower asymptotic density. In particular, in Theorem \ref{aeupperdensitytheorem} upper asymptotic density cannot be replaced by lower asymptotic density.

In addition to the a.e.\ asymptotics, we also find the extreme values and the variance of $\log P_N (\alpha )$ in a distributional sense. Since $P_N(\alpha )$ is $1$-periodic in the variable $\alpha$, we will choose $\alpha$ randomly from $[0,1]$. The main message of the next two results is that typically $\max_{1 \le N \le M} |\log P_N (\alpha )|$ is greater than $(M^{-1} \sum_{N=1}^M (\log P_N (\alpha ))^2)^{1/2}$ by a factor of $\log \log M$. In particular, in \eqref{aevariance} ``for infinitely many $M$'' cannot be replaced by ``for all $M$''.
\begin{thm}\label{convergenceinmeasuretheorem} We have
\[ \frac{\displaystyle \max_{1 \le N \le M} \log P_N (\alpha )}{\log M \log \log M} \to \frac{12V}{\pi^2} \qquad \textrm{in measure, as } M \to \infty . \]
The same holds with $\max$ replaced by $-\min$.
\end{thm}
\noindent Convergence in measure follows with respect to any probability measure $\mu$ on $[0,1]$ that is absolutely continuous with respect to the Lebesgue measure. In contrast to the extreme values, the suitably normalized variance converges to the standard L\'evy distribution.
\begin{thm}\label{limitdistributiontheorem} We have
\[ \frac{1}{M} \sum_{N=1}^M \log P_N (\alpha ) = \frac{1}{2} \log M + o\left( \sqrt{\log M} (\log \log M)^2 \right) \qquad \textrm{in measure, as } M \to \infty . \]
Further, let $\mu$ be a probability measure on $[0,1]$ that is absolutely continuous with respect to the Lebesgue measure. For any $t \ge 0$,
\[ \mu \bigg( \bigg\{ \alpha \in [0,1] \, : \, \frac{10 \pi}{M (\log M)^2} \sum_{N=1}^M \left( \log P_N (\alpha ) - \frac{1}{2} \log M \right)^2 \le t \bigg\} \bigg) \to \int_0^t \frac{e^{-1/(2x)}}{\sqrt{2 \pi} x^{3/2}} \, \mathrm{d}x \quad \textrm{as } M \to \infty . \]
\end{thm}

The different behavior of a.e.\ $\alpha$ compared to badly approximable irrationals and Euler's number --- anticoncentration instead of concentration, the failure of the central limit theorem and the difference between limsup and liminf --- might come as a surprise, considering that a.e.\ $\alpha$ is only logarithmic factors away from being badly approximable: we have $\inf_{q \ge 1} \varphi (q) \| q \alpha \| >0$ if and only if $\sum_{k=1}^{\infty} 1/\varphi (k) < \infty$. The explanation lies in the continued fraction $\alpha =[a_0;a_1,a_2, \dots ]$. As we will see, $\max_{1 \le N <q_k} \log P_N (\alpha)$ is comparable to $a_1+\cdots +a_k$, while the variance is comparable to $a_1^2+\cdots +a_k^2$. For badly approximable irrationals and Euler's number each term in these sums is negligible compared to the whole sum; in contrast, for a.e.\ $\alpha$ the maximal term can dominate. Our proofs rely on asymptotic results on (trimmed) sums of partial quotients, see \cite{BBH,DV,PH,SA} and references therein.

\subsection{A Diophantine product}\label{diophantineproductsection}

The reader might wonder what happens if in the definition of the Sudler product we replace $|\sin (\pi x)|$ by a similar function such as the distance from the nearest integer function $\| x \|$. The behavior of the Diophantine product $\prod_{n=1}^N \| n \alpha \|$ can actually be reduced to that of the Sudler product.
\begin{prop}\label{diophantineproduct} Assume that $\inf_{q \ge 1} q^c \| q \alpha \| >0$ with some constant $1 \le c<2$. Then
\[ P_N (\alpha ) \ll \prod_{n=1}^N \left( 2e \| n \alpha \| \right) \ll P_N (\alpha ) \]
with implied constants depending only on $\alpha$.
\end{prop}
\noindent The Diophantine condition is satisfied by all irrationals mentioned above: with $c=1$ by a badly approximable $\alpha$, and with $c=1+\varepsilon$ by Euler's number $e$ and by a.e.\ $\alpha$. All results for $P_N (\alpha )$ stated above including the concentration inequalities and the central limit theorems thus have an analogue for the Diophantine product. For instance, \eqref{goldenratio} reads
\[ 1 \ll \prod_{n=1}^N \bigg( 2e \bigg\| n \frac{1+\sqrt{5}}{2} \bigg\| \bigg) \ll N, \]
where both the upper and the lower bound are sharp.
\begin{proof}[Proof of Proposition \ref{diophantineproduct}] Consider the even $1$-periodic functions $f(x)=\log |2 \sin (\pi x)|$ and $g(x)=\log \| x \| +1+\log 2$. Note that $g$ has zero mean, and its Fourier coefficients are easily computable using integration by parts:
\[ \begin{split} \int_0^1 g(x) \cos (2 \pi m x) \, \mathrm{d}x &= 2 \int_0^{1/2} (\log x) \cos (2 \pi m x) \, \mathrm{d}x = -2 \int_0^{1/2} \frac{\sin (2 \pi m x)}{2 \pi m x} \, \mathrm{d}x \\ &= -\frac{1}{\pi m} \int_0^{\pi m} \frac{\sin x}{x} \, \mathrm{d} x = - \frac{1}{2m} +c_m \end{split} \]
with $c_m = 1/(\pi m) \int_{\pi m}^{\infty} (\sin x)/x \, \mathrm{d}x \ll 1/m^2$. In the last step we used $\int_0^{\infty} (\sin x)/x \, \mathrm{d}x = \pi /2$. By sufficient smoothness, the Fourier series of $f(x)$ and $g(x)$ converge pointwise at nonintegral reals, hence
\[ g(x) = \sum_{m=1}^{\infty} \left( - \frac{1}{m}+2c_m \right) \cos (2 \pi m x) = f(x) + \sum_{m=1}^{\infty} 2c_m \cos (2 \pi m x) . \]
Therefore
\[ \left| \sum_{n=1}^N g(n \alpha ) - \sum_{n=1}^N f(n\alpha ) \right| = \left| \sum_{m=1}^{\infty} c_m \left( 1- \frac{\sin (\pi (2N+1)m \alpha)}{\sin (\pi m \alpha )} \right) \right| \ll \sum_{m=1}^{\infty} \frac{1}{m^2 \| m \alpha \|} . \]
The last series is easily seen to converge under the assumption $\inf_{q \ge 1} |q|^{c} \| q \alpha \|>0$, $c<2$ (cf.\ \cite[Chapter 2.3]{KN} and Lemma \ref{diophantinelemma} below), which proves the claim.
\end{proof}

\section{Distribution of Sudler products}\label{sudlerdistribution}

\subsection{Expected value and variance}

Our results on the expected value and the variance of Sudler products will be based on the explicit formulas
\begin{equation}\label{explicit1}
\log P_N(\alpha ) = \sum_{m=1}^{\infty} \frac{1}{2m} \left( 1-\frac{\sin (\pi (2N+1)m\alpha )}{\sin (\pi m \alpha )} \right)
\end{equation}
and
\begin{equation}\label{explicit2}
\frac{1}{M} \sum_{N=0}^{M-1} \log P_N(\alpha ) = \sum_{m=1}^{\infty} \frac{1}{2m} \left( 1-\frac{\sin^2 (\pi M m\alpha )}{M\sin^2 (\pi m \alpha )} \right) ,
\end{equation}
which follow directly from the Fourier series expansion \eqref{fourierseries} for all irrational $\alpha$. We sum over $0 \le N \le M-1$ instead of $1 \le N \le M$ purely for aesthetic reasons, and to emphasize the connection to the Dirichlet kernel $\frac{\sin (\pi (2N+1)x)}{\sin (\pi x)}$ and the Fej\'er kernel $\frac{\sin^2 (\pi M x )}{M\sin^2 (\pi x)}$.

Let $\alpha =[a_0;a_1,a_2,\dots ]$ be irrational with convergents $p_k/q_k=[a_0;a_1, \dots, a_k]$. For the sake of simplicity, we will assume that the partial quotients increase at most polynomially fast, i.e.\ $a_k \ll k^d$ with some constant $d \ge 0$. This covers all types of irrationals mentioned before: badly approximable irrationals with $d=0$, Euler's number with $d=1$ and a.e.\ irrational with $d=1+\varepsilon$. This condition is closely related to the irrationality measure of $\alpha$. For instance, we have the implications
\[ a_k \ll k^d \,\, \Longrightarrow \,\, \inf_{q \ge 2} \| q \alpha \| q (\log q)^d >0 \,\, \Longrightarrow \,\, a_k \ll (k \log k)^d . \]
Obvious modifications of the proof yield similar results for more general irrationals.
\begin{prop}\label{momentsprop}
Assume that $a_k \le ck^d$ for all $k \ge 1$ with some constants $c,d \ge 0$. For any $q_k \le M < q_{k+1}$ we have
\begin{equation}\label{expectedvalue}
\frac{1}{M} \sum_{N=0}^{M-1} \log P_N (\alpha ) = \frac{1}{2} \log M + O \left( \max_{|\ell -k| \ll \log k} a_{\ell} \cdot \log \log M \right)
\end{equation}
and
\begin{equation}\label{variance}
\frac{1}{M} \sum_{N=0}^{M-1} \left( \log P_N (\alpha ) - \frac{1}{2} \log M \right)^2 = \sum_{m=1}^M \frac{1}{8 \pi^2 m^2 \| m \alpha \|^2} + O \left( \max_{|\ell -k|\ll \log k} a_{\ell}^2 \cdot (\log \log M)^4 \right)
\end{equation}
with implied constants depending only on $c$ and $d$.
\end{prop}
\noindent It remains to analyze the Diophantine sum on the right hand side of \eqref{variance}, which we do for various types of irrationals in Sections \ref{badlysection} and \ref{nonbadlysection}. Following the methods of Beck \cite[Proposition 3.1, p.\ 171]{BE1}, we first prove several simple Diophantine sum estimates, and then give the proof of Proposition \ref{momentsprop}.

\begin{lem}\label{diophantinelemma} Assume that $a_k \le ck^d$ for all $k \ge 1$ with some constants $c,d \ge 0$. For any $q_k \le M < q_{k+1}$,
\[ \begin{split} \sum_{1 \le m \le M} \frac{1}{\| m \alpha \|} &\ll M (\log M)^{d+1} , \\ \sum_{1 \le m \le M (\log M)^{2d+2}} \frac{1}{m} \min \left\{ \frac{1}{M \| m \alpha \|^2}, M \right\} &\ll \max_{|\ell -k| \ll \log k} a_{\ell} \cdot \log \log M , \\ \sum_{1 \le m \le M (\log M)^{2d+2}} \frac{1}{m^2 \| m \alpha \|^2} \min \left\{ \frac{1}{M \| 2 m \alpha \|} ,1 \right\} &\ll \max_{|\ell -k| \ll \log k} a_{\ell}^2 \cdot \log \log M , \\ \sum_{M<m \le M (\log M)^{2d+2}} \frac{1}{m^2 \| m\alpha \|^2} &\ll \max_{|\ell -k| \ll \log k} a_{\ell}^2 \cdot \log \log M . \end{split} \]
Further,
\[ \sum_{\substack{1 \le m_1, m_2 \le M (\log M)^{2d+2} \\ m_1 \neq m_2}} \frac{1}{m_1 \| m_1 \alpha \| m_2 \| m_2 \alpha \|} \min \left\{ \frac{1}{M \| (m_1 - m_2) \alpha \|} ,1 \right\} \ll \max_{|\ell -k| \ll \log k} a_{\ell}^2 \cdot (\log \log M)^4 , \]
and the same holds with $m_1-m_2$ replaced by $m_1+m_2$.
\end{lem}

\begin{proof} Since $q_k$ increases at least as fast as the sequence of Fibonacci numbers, we have $k \ll \log q_k \ll \log M$. We will also use the simple fact $q_0+q_1+\cdots +q_{\ell} \le 3 q_{\ell}$ for all $\ell \ge 0$. This follows e.g.\ from the recursion satisfied by $q_{\ell}$ via
\[ q_1 + q_2 + \cdots +q_{\ell-1} \le a_2 q_1 + a_3 q_2 + \cdots + a_{\ell} q_{\ell-1} = q_{\ell} + q_{\ell-1} -q_1-q_0 \le 2q_{\ell}-q_0. \]

All five claims follow from a classical method based on the pigeonhole principle, see \cite[Chapter 2.3]{KN}. More specifically, let $\ell \ge 0$, and consider the points $m \alpha \pmod{1}$, $q_{\ell} \le m < q_{\ell +1}$. By the best rational approximation property, $\| m \alpha \| \ge \| q_{\ell} \alpha \| \ge 1/(2q_{\ell +1})$. In particular, none of these points lie in the open interval $(-1/(2q_{\ell +1}), 1/(2q_{\ell +1}))$, and each interval $(j/(2q_{\ell +1}), (j+1)/(2q_{\ell +1}))$, $j \neq -1,0$ contains at most one point.

To see the first claim of the lemma, observe that for any $0 \le \ell \le k$,
\[ \sum_{q_{\ell} \le m < \min \{ q_{\ell +1},M+1\}} \frac{1}{\| m \alpha \|} \ll \sum_{j=1}^M \frac{1}{j/q_{\ell +1}} \ll q_{\ell +1} \log M . \]
Summing over $0 \le \ell \le k$ leads to
\[ \sum_{m=1}^M \frac{1}{\| m \alpha \|} \ll q_{k+1} \log M \ll k^d q_k \log M \ll M (\log M)^{d+1}. \]

We now prove the second claim of the lemma. For any $\ell \ge 0$ with $q_{\ell} \le M/(\log M)^{2d}$,
\[ \sum_{q_{\ell} \le m < q_{\ell +1}} \frac{1}{m} \min \left\{ \frac{1}{M \| m \alpha \|^2}, M \right\} \le \frac{1}{q_{\ell} M} \sum_{q_{\ell} \le m < q_{\ell +1}} \frac{1}{\| m \alpha \|^2} \ll \frac{1}{q_{\ell} M} \sum_{j=1}^{\infty} \frac{1}{j^2 / q_{\ell +1}^2} \ll \frac{a_{\ell +1}^2 q_{\ell}}{M} . \]
By the growth assumption on the partial quotients, this is $\ll q_{\ell} (\log M)^{2d} /M$. Next, note that there are $\ll \log \log M$ indices $\ell$ with $M/(\log M)^{2d} < q_{\ell} \le M (\log M)^{2d+2}$, and all of these indices satisfy $|\ell -k| \ll \log \log M \ll \log k$. For any such $\ell$ we have
\[ \sum_{q_{\ell} \le m < q_{\ell +1}} \frac{1}{m} \min \left\{ \frac{1}{M \| m \alpha \|^2}, M \right\} \ll \frac{1}{q_{\ell}} \sum_{j=1}^{\infty} \min \left\{ \frac{1}{M j^2/q_{\ell +1}^2}, M \right\} \ll \frac{q_{\ell +1}}{q_{\ell}} \ll a_{\ell +1} . \]
Therefore
\[ \begin{split} \sum_{1 \le m \le M (\log M)^{2d+2}} &\frac{1}{m} \min \left\{ \frac{1}{M \| m \alpha \|^2}, M \right\} \\ &\le \Big( \sum_{\substack{\ell \ge 0 \\ q_{\ell} \le M/(\log M)^{2d}}} +\sum_{\substack{\ell \ge 0 \\ M/(\log M)^{2d} < q_{\ell} \le  M (\log M)^{2d+2}}} \Big) \sum_{q_{\ell} \le m < q_{\ell +1}} \frac{1}{m} \min \left\{ \frac{1}{M \| m \alpha \|^2}, M \right\} \\ &\ll \sum_{\substack{\ell \ge 0 \\ q_{\ell} \le M/(\log M)^{2d}}} \frac{q_{\ell} (\log M)^{2d}}{M} + \sum_{\substack{\ell \ge 0 \\ M/(\log M)^{2d} < q_{\ell} \le  M (\log M)^{2d+2}}} a_{\ell +1} \\ &\ll 1+ \max_{|\ell -k| \ll \log k} a_{\ell} \cdot \log \log M, \end{split} \]
as claimed.

Let us now prove the third claim of the lemma. By the subadditivity of the function $\| \cdot \|$, we have $\| m \alpha \| \ge \frac{1}{2} \| 2 m \alpha \|$. Hence
\[ \begin{split} \sum_{1 \le m \le M (\log M)^{2d+2}} \frac{1}{m^2 \| m \alpha \|^2} &\min \left\{ \frac{1}{M \| 2 m \alpha \|} ,1 \right\} \\ &\ll \sum_{1 \le m \le M (\log M)^{2d+2}} \frac{1}{(2m)^2 \| 2 m \alpha \|^2} \min \left\{ \frac{1}{M \| 2 m \alpha \|} ,1 \right\} \\ &\ll \sum_{1 \le m \le 2 M (\log M)^{2d+2}} \frac{1}{m^2 \| m \alpha \|^2} \min \left\{ \frac{1}{M \| m \alpha \|} ,1 \right\} , \end{split} \]
and it will be enough to estimate the last line of the previous formula. For any $\ell \ge 0$ with $q_\ell \le M/(\log M)^{3d}$,
\[ \sum_{q_{\ell} \le m < q_{\ell+1}} \frac{1}{m^2 \| m \alpha \|^2} \min \left\{ \frac{1}{M \| m \alpha \|} ,1 \right\} \le \frac{1}{q_{\ell}^2 M} \sum_{q_{\ell} \le m < q_{\ell+1}} \frac{1}{\| m \alpha \|^3} \ll \frac{1}{q_{\ell}^2 M} \sum_{j=1}^{\infty} \frac{1}{j^3/q_{\ell+1}^3} \ll \frac{a_{\ell+1}^3 q_{\ell}}{M}. \]
By the growth assumption on the partial quotients, this is $\ll q_{\ell} (\log M)^{3d}/M$. Next, note that there are $\ll \log \log M$ indices $\ell$ with $M/(\log M)^{3d} < q_{\ell} \le 2M (\log M)^{2d+2}$, and all of these indices satisfy $|\ell -k| \ll \log \log M \ll \log k$. For any such $\ell$ we have
\[ \sum_{q_{\ell} \le m < q_{\ell+1}} \frac{1}{m^2 \| m \alpha \|^2} \min \left\{ \frac{1}{M \| m \alpha \|} ,1 \right\} \ll \frac{1}{q_{\ell}^2} \sum_{j=1}^{\infty} \frac{1}{j^2/q_{\ell+1}^2} \ll \frac{q_{\ell+1}^2}{q_{\ell}^2} \ll a_{\ell +1}^2 . \]
Therefore
\[ \begin{split} \sum_{1 \le m \le 2 M (\log M)^{2d+2}} &\frac{1}{m^2 \| m \alpha \|^2} \min \left\{ \frac{1}{M \| m \alpha \|} ,1 \right\} \\ &\le \Big( \sum_{\substack{\ell \ge 0 \\ q_{\ell} \le M/(\log M)^{3d}}} +\sum_{\substack{\ell \ge 0 \\ M/(\log M)^{3d} < q_{\ell} \le 2M (\log M)^{2d+2}}} \Big) \frac{1}{m^2 \| m \alpha \|^2} \min \left\{ \frac{1}{M \| m \alpha \|} ,1 \right\} \\ &\ll \sum_{\substack{\ell \ge 0 \\ q_{\ell} \le M/(\log M)^{3d}}} \frac{q_{\ell} (\log M)^{3d}}{M} + \sum_{\substack{\ell \ge 0 \\ M/(\log M)^{3d} < q_{\ell} \le 2M (\log M)^{2d+2}}} a_{\ell+1}^2 \\ &\ll 1+\max_{|\ell -k| \ll \log k} a_{\ell}^2 \cdot \log \log M, \end{split} \]
as claimed.

To see the fourth claim of the lemma, observe that for any $\ell \ge 0$,
\[ \sum_{q_{\ell} \le m < q_{\ell+1}} \frac{1}{m^2 \| m \alpha \|^2} \ll \frac{1}{q_{\ell}^2} \sum_{j=1}^{\infty} \frac{1}{j^2/q_{\ell+1}^2} \ll \frac{q_{\ell+1}^2}{q_{\ell}^2} \ll a_{\ell +1}^2. \]
There are $\ll \log \log M$ indices $\ell$ for which the intervals $[q_{\ell}, q_{\ell+1})$ and $[M,M (\log M)^{2d+2}]$ intersect, and all of these indices satisfy $|\ell -k| \ll \log \log M \ll \log k$. Summing over all such indices $\ell$ leads to
\[ \sum_{M \le m \le M (\log M)^{2d+2}} \frac{1}{m^2 \| m \alpha \|^2} \ll \sum_{\substack{\ell \ge 0 \\ [q_{\ell}, q_{\ell+1}) \cap [M,M (\log M)^{2d+2}] \neq \emptyset}} a_{\ell+1}^2 \ll \max_{|\ell-k|\ll \log k} a_{\ell}^2 \cdot \log \log M. \]

We now prove the final claim. Here the cases $m_1-m_2$ and $m_1+m_2$ are analogous, and we will only consider the $m_1-m_2$ case. Let
\[ R_M:= \sum_{\substack{1 \le m_1, m_2 \le M (\log M)^{2d+2} \\ m_1 \neq m_2}} \frac{1}{m_1 \| m_1 \alpha \| m_2 \| m_2 \alpha \|} \min \left\{ \frac{1}{M \| (m_1 - m_2) \alpha \|} ,1 \right\} \]
denote the sum to be estimated. Using similar methods as before, we deduce
\[ \sum_{1 \le m \le M (\log M)^{2d+2}} \frac{1}{m \| m \alpha \|} \ll \sum_{\substack{\ell \ge 0 \\ q_{\ell} \le M (\log M)^{2d+2}}} \frac{1}{q_{\ell}} \sum_{1 \le j \le M (\log M)^{2d+2}} \frac{1}{j/q_{\ell +1}} \ll (\log M)^{d+2} . \]
We will also need the fact that $\| m \alpha \| \gg 1/(m (\log m)^d)$ for all $m>1$, which follows from the assumption on the growth rate of the partial quotients.

Let $S=\{ m \ge 1 \, : \, m \| m \alpha \| \ge (\log M)^{4d+5} \}$. We first observe that the terms with $m_1 \in S$ are negligible:
\begin{multline*}
\sum_{\substack{1 \le m_1, m_2 \le M (\log M)^{2d+2} \\ m_1 \neq m_2 \\ m_1 \in S}} \frac{1}{m_1 \| m_1 \alpha \| m_2 \| m_2 \alpha \|} \min \left\{ \frac{1}{M \| (m_1 - m_2) \alpha \|} ,1 \right\} \\ \ll \frac{1}{(\log M)^{4d+5}} \sum_{1 \le m_2 \le M (\log M)^{2d+2}} \frac{1}{m_2 \| m_2 \alpha \|} \sum_{1 \le m' \le M (\log M)^{2d+2}} \frac{1}{M \| m' \alpha \|} \ll 1.
\end{multline*}
Note that instead of $m_1$ and $m_2$, we summed over $m':=|m_1-m_2|$ and $m_2$. By symmetry, the same holds for the terms with $m_2 \in S$. We similarly see that the terms with $\| (m_1-m_2) \alpha \| \ge (\log M)^{2d+4}/M$ are also negligible:
\begin{multline*}
\sum_{\substack{1 \le m_1, m_2 \le M (\log M)^{2d+2} \\ \| (m_1-m_2) \alpha \| \ge (\log M)^{2d+4}/M}} \frac{1}{m_1 \| m_1 \alpha \| m_2 \| m_2 \alpha \|} \min \left\{ \frac{1}{M \| (m_1 - m_2) \alpha \|} ,1 \right\} \\ \ll \frac{1}{(\log M)^{2d+4}} \sum_{1 \le m_1 \le M (\log M)^{2d+2}} \frac{1}{m_1 \| m_1 \alpha \|} \sum_{1 \le m_2 \le M (\log M)^{2d+2}} \frac{1}{m_2 \| m_2 \alpha \|} \ll 1.
\end{multline*}
By the previous two formulas,
\[ R_M \ll 1+ \sum_{\substack{1 \le m_1, m_2 \le M (\log M)^{2d+2} \\ 0<\| (m_1-m_2) \alpha \| \le (\log M)^{2d+4}/M \\ m_1, m_2 \not\in S}} \frac{1}{m_1 \| m_1 \alpha \| m_2 \| m_2 \alpha \|} . \]
We now claim that in the remaining sum there is no term with $1 \le m_1 \le M/(\log M)^{8d+10}$. Indeed, such an $m_1$ would satisfy $\| m_1 \alpha \| \gg (\log M)^{7d+10}/M$. On the other hand, the condition $0<\| (m_1-m_2) \alpha \| \le (\log M)^{2d+4}/M$ forces $|m_1-m_2| \gg M/(\log M)^{3d+4}$. This in turn implies $m_2 \gg M/(\log M)^{3d+4}$ and, since $m_2 \not\in S$,
\[ \| m_2 \alpha \| \le \frac{(\log M)^{4d+5}}{m_2} \ll \frac{(\log M)^{7d+9}}{M} . \]
In particular, for large enough $M$ we have, say, $\| m_1 \alpha \| \ge 2 \| m_2 \alpha \|$ and hence
\[ \| (m_1-m_2) \alpha \| \ge \| m_1 \alpha \| - \| m_2 \alpha \| \ge \frac{1}{2} \| m_1 \alpha \| \gg \frac{(\log M)^{7d+10}}{M} , \]
giving a contradiction. By symmetry, there are no terms with $1 \le m_2 \le M/(\log M)^{8d+10}$ either, therefore
\[ R_M \ll 1+ \bigg( \sum_{\substack{M/(\log M)^{8d+10} \le m \le M (\log M)^{2d+2} \\ m \not\in S}} \frac{1}{m \| m \alpha \|} \bigg)^2 . \]
To estimate the remaining sum, let $\ell \ge 0$ be an index such that the intervals $[q_{\ell}, q_{\ell +1})$ and $[M/(\log M)^{8d+10}, M (\log M)^{2d+2}]$ intersect. Note that there are $\ll \log \log M$ such indices $\ell$, each of which satisfies $|\ell -k| \ll \log k$. For any such $\ell$,
\[ \sum_{\substack{q_{\ell} \le m <q_{\ell +1} \\ m \not\in S}} \frac{1}{m \| m \alpha \|} \le \frac{1}{q_{\ell}} \sum_{1 \le j \ll (\log M)^{5d+5}} \frac{1}{j/q_{\ell +1}} \ll a_{\ell +1} \log \log M .  \]
Note that $m \not\in S$ implies that
\[ \| m \alpha \| \le \frac{(\log M)^{4d+5}}{m} \le \frac{(\log M)^{4d+5}}{q_{\ell}} , \]
explaining why we only sum over $j$ for which $j/q_{\ell +1} \le (\log M)^{4d+5}/q_{\ell}$, that is, for which $j \ll (q_{\ell +1}/q_{\ell}) (\log M)^{4d+5} \ll (\log M)^{5d+5}$. Hence
\[ \begin{split} \sum_{\substack{M/(\log M)^{8d+10} \le m \le M (\log M)^{2d+2} \\ m \not\in S}} \frac{1}{m \| m \alpha \|} &\ll \sum_{\substack{\ell \ge 0 \\ [q_{\ell}, q_{\ell +1}) \cap [M/(\log M)^{8d+10}, M (\log M)^{2d+2}] \neq \emptyset}} a_{\ell +1} \log \log M \\ &\ll \max_{|\ell -k| \ll \log k} a_{\ell} \cdot (\log \log M)^2 , \end{split} \]
and the claim follows.
\end{proof}

\begin{proof}[Proof of Proposition \ref{momentsprop}] Fix $q_k \le M < q_{k+1}$, and let $0 \le N \le M-1$. We start by estimating the tails in the explicit formulas \eqref{explicit1} and \eqref{explicit2}. Summation by parts gives that the tails of the Fourier series expansion \eqref{fourierseries} decay at the rate
\begin{equation}\label{Htail}
\left| \sum_{m=H}^{\infty} \frac{\cos (2 \pi m x)}{m} \right| \le \frac{1}{2 H \| x \|} , \qquad H \ge 1.
\end{equation}
Choosing $H=\lceil M (\log M)^{2d+2} \rceil$ we deduce
\begin{equation}\label{logpnexpansion}
\begin{split} \log P_N (\alpha ) &= \sum_{n=1}^N \sum_{1 \le m \le M (\log M)^{2d+2}} \frac{-\cos (2 \pi m n \alpha )}{m} + O \left( \frac{1}{M (\log M)^{2d+2}} \sum_{n=1}^N \frac{1}{\| n \alpha \|} \right) \\ &= \sum_{1 \le m \le M (\log M)^{2d+2}} \frac{1}{2m} \left( 1-\frac{\sin (\pi (2N+1)m\alpha)}{\sin (\pi m \alpha)} \right) + O \left( \frac{1}{(\log M)^{d+1}} \right) , \end{split}
\end{equation}
and by averaging over $0 \le N \le M-1$,
\begin{equation}\label{logpnaverageexpansion}
\frac{1}{M} \sum_{N=0}^{M-1} \log P_N (\alpha ) = \sum_{1 \le m \le M (\log M)^{2d+2}} \frac{1}{2m} \left( 1-\frac{\sin^2 (\pi M m\alpha)}{M\sin^2 (\pi m \alpha)} \right) + O \left( \frac{1}{(\log M)^{d+1}} \right) .
\end{equation}
Note that in \eqref{logpnexpansion} we used the first Diophantine sum estimate from Lemma \ref{diophantinelemma}.

Using the estimate
\[ \frac{\sin^2 (\pi M m\alpha)}{M\sin^2 (\pi m \alpha)} \le \min \left\{ \frac{1}{M \| m \alpha \|^2}, M \right\} \]
in \eqref{logpnaverageexpansion}, we deduce
\[ \frac{1}{M} \sum_{N=0}^{M-1} \log P_N (\alpha ) = \frac{1}{2} \log M + O \bigg( \log \log M+ \sum_{1 \le m \le M (\log M)^{2d+2}} \frac{1}{m} \min \left\{ \frac{1}{M \| m \alpha \|^2}, M \right\} \bigg) . \]
Claim \eqref{expectedvalue} thus follows from the second Diophantine sum estimate in Lemma \ref{diophantinelemma}.

Next, we prove \eqref{variance}. Formula \eqref{logpnexpansion} yields
\begin{equation}\label{varianceVM}
\frac{1}{M} \sum_{N=0}^{M-1} \left( \log P_N (\alpha ) - H_M \right)^2 = V_M + O \left( \frac{\sqrt{V_M}}{(\log M)^{d+1}} + \frac{1}{(\log M)^{2d+2}} \right)
\end{equation}
with
\[ H_M:=\sum_{1 \le m \le M (\log M)^{2d+2}} \frac{1}{2m} = \frac{1}{2} \log M + O(\log \log M) \]
and
\[ V_M:=\frac{1}{M} \sum_{N=0}^{M-1} \left( \sum_{1 \le m \le M (\log M)^{2d+2}} \frac{\sin (\pi (2N+1)m\alpha)}{2m \sin (\pi m \alpha)} \right)^2 . \]
Let us now expand the square in the previous formula. The diagonal terms satisfy
\[ \begin{split} \frac{1}{M} \sum_{N=0}^{M-1} \sum_{1 \le m \le M (\log M)^{2d+2}} &\frac{\sin^2 (\pi (2N+1)m\alpha)}{4 m^2 \sin^2 (\pi m \alpha )} \\ &= \sum_{1 \le m \le M (\log M)^{2d+2}} \frac{1}{4 m^2 \sin^2 (\pi m \alpha )} \left( \frac{1}{2}+O \left( \min \left\{ \frac{1}{M \| 2 m \alpha \|} ,1 \right\} \right) \right) \\ &= \sum_{m=1}^M \frac{1}{8 \pi^2 m^2 \| m \alpha \|^2} + O \left( \max_{|\ell -k| \ll \log k} a_{\ell}^2 \cdot \log \log M \right) , \end{split} \]
where we used the third claim in Lemma \ref{diophantinelemma} and the fact that $1/\sin^2 (\pi x) = 1/(\pi^2 \| x \|^2) +O(1)$. The last claim in Lemma \ref{diophantinelemma} also shows that the contribution of the off-diagonal terms is negligible:
\[ \begin{split} \bigg| \frac{1}{M} \sum_{N=0}^{M-1} &\sum_{\substack{1 \le m_1, m_2 \le M (\log M)^{2d+2} \\ m_1 \neq m_2}} \frac{\sin (\pi (2N+1)m_1\alpha ) \sin (\pi (2N+1)m_2 \alpha )}{2m_1 \sin (\pi m_1 \alpha ) 2m_2 \sin (\pi m_2 \alpha )} \bigg| \\ \ll &\sum_{\substack{1 \le m_1, m_2 \le M (\log M)^{2d+2} \\ m_1 \neq m_2}} \frac{1}{m_1 \| m_1 \alpha \| m_2 \| m_2 \alpha \|} \min \left\{ \frac{1}{M \| (m_1 - m_2) \alpha \|} ,1 \right\} \\ &+\sum_{\substack{1 \le m_1, m_2 \le M (\log M)^{2d+2} \\ m_1 \neq m_2}} \frac{1}{m_1 \| m_1 \alpha \| m_2 \| m_2 \alpha \|} \min \left\{ \frac{1}{M \| (m_1 + m_2) \alpha \|} ,1 \right\} \\ \ll &\max_{|\ell -k| \ll \log k} a_{\ell}^2 \cdot (\log \log M)^4 . \end{split} \]
By the previous two formulas,
\[ V_M= \sum_{m=1}^M \frac{1}{8 \pi^2 m^2 \| m\alpha \|^2} + O \left( \max_{|\ell -k| \ll \log k} a_{\ell}^2 \cdot (\log \log M)^4 \right) . \]
In particular, $V_M \ll (\log M)^{2d+1}$ (cf.\ formula \eqref{diophantinesumpartialquotients} below), hence the error term on the right hand side of \eqref{varianceVM} is $O(1)$. The first claim \eqref{expectedvalue} shows that the error of replacing $H_M$ by $(1/2) \log M$ in \eqref{varianceVM} is also negligible, and the claim \eqref{variance} follows.
\end{proof}

\subsection{Badly approximable irrationals}\label{badlysection}

Fix a badly approximable $\alpha$. We first note that \eqref{expectedbadlyapprox} and \eqref{variancebadlyapprox} follow easily from Proposition \ref{momentsprop}. Indeed, the only missing piece is the observation
\begin{equation}\label{m2malpha2sum}
\log M \ll \sum_{m=1}^M \frac{1}{8 \pi^2 m^2 \| m \alpha \|^2} \ll \log M .
\end{equation}
Here the upper bound follows from the pigeonhole principle as in Lemma \ref{diophantinelemma} (cf.\ formula \eqref{diophantinesumpartialquotients} below and \cite[Theorem 3]{BO}). To see the lower bound, simply keep the convergent denominators $m=q_k \le M$: each such term contributes $\gg 1$, and there are $\gg \log M$ denominators $q_k \le M$.

\begin{proof}[Proof of Theorem \ref{badlyapproxtheorem}] From \eqref{expectedbadlyapprox}, \eqref{variancebadlyapprox} and the Chebyshev inequality we obtain
\[ \frac{1}{M} \left| \left\{ 1 \le N \le M \, : \, \left| \log P_N (\alpha ) - \frac{1}{2} \log M \right| \ge t \sqrt{\log M} \right\} \right| \ll \frac{1}{t^2} \]
with an implied constant depending only on $\alpha$. It remains to replace $\log M$ by $\log N$ in the previous formula. For any $M/(\log M)^2 \le N \le M$ we have $\log N = \log M+O(\log \log M)$, and so one readily checks that
\[ \frac{1}{M} \left| \left\{ 1 \le N \le M \, : \, \left| \log P_N (\alpha ) - \frac{1}{2} \log N \right| \ge t \sqrt{\log N} \right\} \right| \ll \frac{1}{(\log M)^2} + \frac{1}{t^2} . \]
Note that $1/(\log M)^2$ is the ``probability'' of $1 \le N \le M/(\log M)^2$. The claim thus follows once $t \ll \log M$. However, the claim also trivially holds for $t \gg \log M$, since by the result of Lubinsky, $|\log P_N (\alpha )| \ll \log N$ for all $N$.
\end{proof}

Next, assume in addition that $\alpha$ is a quadratic irrational. Using deep arithmetic properties of the real quadratic field $\mathbb{Q}(\alpha)$, Beck \cite[Proposition 3.2, p.\ 176]{BE1} proved that in this case
\begin{equation}\label{beckformula}
\sum_{m=1}^M \frac{1}{8\pi^2 m^2 \| m \alpha \|^2} = \sigma (\alpha )^2 \log M +O(1)
\end{equation}
with some constant $\sigma (\alpha) >0$ and an implied constant depending only on $\alpha$. Formula \eqref{variancequadratic} thus follows from Proposition \ref{momentsprop}.

\subsection{Central limit theorem for quadratic irrationals}\label{CLTsection}

Theorem \ref{CLTtheorem} is obviously a special case of Theorem \ref{2dimCLTtheorem}, so it will be enough to prove the latter. We will actually work with $0 \le N < M$ instead of $1 \le N \le M$, but this of course does not affect the result. The reason is simply notational convenience, as the former interval is better suited for Ostrowski expansions.

We follow the ``Fourier series approach'' of Beck \cite[Chapter 4.4]{BE1}. There it is shown that for a quadratic irrational $\alpha$, the Birkhoff sum $S_N(\alpha ) = \sum_{n=1}^N (\{ n \alpha \} -1/2)$ can be written in the form
\[ S_N (\alpha ) - E(\alpha ) \log M = \sum_{1 \le m \le M \log M} \frac{\cos (\pi (2N+1)m\alpha )}{2 \pi m \sin (\pi m \alpha )} +O(\log \log M), \quad 0 \le N < M. \]
Using the tail estimate \eqref{Htail} with $H = \lceil M \log M \rceil$ and following the steps in \eqref{logpnexpansion}, we obtain the similar formula
\[ \log P_N (\alpha ) - \frac{1}{2} \log M = \sum_{1 \le m \le M \log M} \frac{-\sin (\pi (2N+1)m \alpha )}{2m \sin (\pi m \alpha )} + O(\log \log M), \quad 0 \le N < M . \]
For the sake of readability, let us use the natural identification $(x,y) \leftrightarrow x+iy$ between $\mathbb{R}^2$ and $\mathbb{C}$. Then, for $0 \le N < M$,
\begin{equation}\label{identification}
\left( \log P_N(\alpha) - \frac{1}{2} \log M, \pi (S_N(\alpha) - E(\alpha) \log M) \right) \leftrightarrow \sum_{1 \le m \le M \log M} \frac{i e^{2 \pi i (N+1/2)m \alpha}}{2m \sin (\pi m \alpha)} + O(\log \log M) .
\end{equation}

Now let $\alpha=[0;a,a,a,\ldots ]$. We will prove that given any $\xi \in \mathbb{R}$, with $M=q_K$ a convergent denominator and any $q_K \ll Q_K \ll q_K$ with implied constants depending only on $\alpha$,
\[ T_N :=\sum_{1 \le m \le Q_K \log Q_K} \frac{e^{2 \pi i (N+\xi)m \alpha}}{2m \sin (\pi m \alpha)}, \quad 0 \le N < q_K \]
satisfies the central limit theorem: for any convex set $C \subseteq \mathbb{C}$,
\begin{equation}\label{TNCLT}
\frac{1}{q_K} \left| \left\{ 0 \le N <q_K \, : \, \frac{T_N}{\sigma(\alpha) \sqrt{\log q_K}} \in C \right\} \right| = \int_C \frac{e^{-|y|^2/2}}{2 \pi} \, \mathrm{d} y + O \left( \frac{(\log K)^2}{K^{1/6}} \right)
\end{equation}
with an implied constant depending only on $\alpha$.

We prove \eqref{TNCLT} in Sections \ref{ostrowskisection}--\ref{independentsection}. Throughout the proof, constants and implied constants depend only on $\alpha$; in particular, every estimate is uniform in $\xi$ and $Q_K$. We show how \eqref{TNCLT} implies Theorem \ref{2dimCLTtheorem} in Section \ref{CLTproofsection}. We mention in advance that the parameters $\xi$ and $Q_K$ will play a technical role in extending the result from $M=q_K$ (a convergent denominator) to general $M$.

\subsubsection{Ostrowski expansion and $\alpha$-expansion}\label{ostrowskisection}

Fix an integer $a \ge 1$, and let $\alpha=[0;a,a,a,\ldots]$. We have
\[ \alpha = \frac{\sqrt{a^2+4}-a}{2} \in (0,1) \qquad \textrm{and} \qquad \frac{1}{\alpha} = \frac{\sqrt{a^2+4}+a}{2} \in (a,a+1) . \]
The minimal polynomial of $\alpha$ gives the useful relation $\alpha^2+a\alpha-1=0$. By solving the recursions, we deduce that the convergents $p_k/q_k$ to $\alpha$ are
\[ p_k= \frac{\alpha}{\alpha^2+1} \left( \frac{1}{\alpha^k} - (-\alpha)^k \right) \qquad \textrm{and} \qquad q_k= \frac{\alpha}{\alpha^2+1} \left( \frac{1}{\alpha^{k+1}} - (-\alpha)^{k+1} \right) . \]
Hence $q_k \alpha - p_k=(-1)^k \alpha^{k+1}$.

The so-called Ostrowski expansion of an integer $0 \le N<q_K$ is the unique representation of the form $N=\sum_{k=0}^{K-1} b_k q_k$, where $0 \le b_0 \le a-1$ and $0 \le b_k \le a \,\, (1 \le k \le K-1)$ are integers which satisfy the extra rule that $b_k=a$ implies $b_{k-1}=0$. The Ostrowski expansion can be found with a greedy algorithm: define $b_{K-1}$ as the largest integer $b$ such that $N \ge b q_{K-1}$; this is obviously an integer in $[0,a]$. Then $0 \le N-b_{K-1}q_{K-1}<q_{K-1}$, and we iterate the same procedure. Note that if $b_{K-1}=a$, then $0 \le N -b_{K-1}q_{K-1}<q_K-aq_{K-1}=q_{K-2}$, hence $b_{K-2}=0$, explaining the extra rule.

The so-called $\alpha$-expansion of a real number $0 \le x <\alpha^{-K}$ is entirely analogous: it is the (almost) unique representation of the form $x=\sum_{k=-\infty}^{K-1} c_k \alpha^{-k}$, where $0 \le c_k \le a$ are integers which satisfy the extra rule that $c_k=a$ implies $c_{k-1}=0$. The $\alpha$-expansion can also be found with a greedy algorithm: define $c_{K-1}$ as the largest integer $c$ such that $x \ge c\alpha^{-(K-1)}$; this is obviously an integer in $[0,a]$. Then $0 \le x-c_{K-1} \alpha^{-(K-1)}<\alpha^{-(K-1)}$, and we iterate the same procedure. Note that if $c_{K-1}=a$, then $0 \le x-c_{K-1}\alpha^{-(K-1)}<\alpha^{-K}-a \alpha^{-(K-1)}=\alpha^{-(K-2)}$, hence $c_{K-2}=0$, explaining the extra rule. Exponentially fast convergence ensures that the infinite series $\sum_{k=-\infty}^{K-1} c_k \alpha^{-k}$ indeed represents $x$.

The $\alpha$-expansion can be visualized by partitioning $[0,\alpha^{-K})$ into $a+1$ intervals corresponding to the possible values of $c_{K-1}$: the first $a$ intervals are $[c\alpha^{-(K-1)}, (c+1)\alpha^{-(K-1)})$, $0 \le c \le a-1$, each of length $\alpha^{-(K-1)}$, and the last interval is $[a \alpha^{-(K-1)}, \alpha^{-K})$, which is of length $\alpha^{-K}-a\alpha^{-(K-1)} = \alpha^{-(K-2)}$. The special length of the last interval represents self-similarity, and is related to the fact that the continued fraction of $\alpha$ has period $1$. Iterating the procedure thus corresponds to a partitioning system of nested intervals. Note that the $\alpha$-expansion is not unique at the endpoints of these nested intervals, but it is in fact unique everywhere else. Since we will eventually work with a real number chosen from an interval uniformly at random, we can ignore this effect.

The following lemma constructs a simple coupling between these two expansions.
\begin{lem}\label{ostrowskilemma} For any real $x \in [0,\alpha^{-K})$, consider the $\alpha$-expansion $x=\sum_{k=-\infty}^{K-1} c_k \alpha^{-k}$, and the Ostrowski expansion $\lfloor x/(1+\alpha^2 ) \rfloor =\sum_{k=0}^{K-1} b_k q_k$. There exists a set $E \subset \mathbb{R}$ of Lebesgue measure $\lambda (E) \ll K^{-10} \alpha^{-K}$ such that for any $x \in [0,\alpha^{-K}) \backslash E$, we have $b_k=c_k$ for all $\lfloor \frac{10 \log K}{\log (1/\alpha)} \rfloor \le k \le K-1$.
\end{lem}

\begin{remark} Note that $0 \le x < \alpha^{-K}$ implies $0 \le x/(1+\alpha^2) < \alpha^{-K}/(1+\alpha^2) = q_K +O(\alpha^K)$. By adding an interval of length $\ll \alpha^K$ to $E$, we may assume that $0 \le \lfloor x/(1+\alpha^2) \rfloor < q_K$.
\end{remark}

\begin{proof}[Proof of Lemma \ref{ostrowskilemma}] Set $k_0=\lfloor \frac{10\log K}{\log (1/\alpha)} \rfloor$. Let $A$ be the set of all integers $N \in [0,q_K)$ whose Ostrowski expansion is of the form $N=\sum_{k=k_0}^{K-1} d_k q_k$, i.e.\ the first $k_0$ digits are all zero. We will show that
\[ E=(1+\alpha^2) \bigcup_{N \in A} \left( N-1,N+1 \right) \cup \left( N+q_{k_0-1}-1, N+q_{k_0-1}+1 \right) \]
satisfies the claim of the lemma, where $(1+\alpha^2) H = \{ (1+\alpha^2) y \, : \, y \in H \}$. The cardinality of $A$ is the number of legitimate Ostrowski sequences $(d_{K-1}, d_{K-2}, \ldots, d_{k_0})$ of length $K-k_0$, hence $|A| \ll q_{K-k_0} \ll \alpha^{k_0 -K} \ll K^{-10} \alpha^{-K}$. Consequently, $\lambda (E) \ll K^{-10} \alpha^{-K}$.

Now let $x \in [0,\alpha^{-K}) \backslash E$. Assume first that $0 \le c_{k_0}<a$. By the construction of the $\alpha$-expansion,
\[ \sum_{k=k_0}^{K-1} c_k \frac{\alpha^{-k}}{1+\alpha^2} \le \frac{x}{1+\alpha^2} < (c_{k_0}+1) \frac{\alpha^{-k_0}}{1+\alpha^2} + \sum_{k=k_0+1}^{K-1} c_k \frac{\alpha^{-k}}{1+\alpha^2} . \]
Replacing $\frac{\alpha^{-k}}{1+\alpha^2}$ by $q_k=\frac{\alpha^{-k}}{1+\alpha^2} +O(\alpha^k)$ introduces an error $O(\alpha^{k_0})$, therefore very roughly,
\[ \sum_{k=k_0}^{K-1} c_k q_k -\frac{1}{100} \le \frac{x}{1+\alpha^2} < (c_{k_0}+1) q_{k_0} + \sum_{k=k_0+1}^{K-1} c_k q_k +\frac{1}{100} . \]
By the assumption $x \not\in E$, this can be improved to
\[ \sum_{k=k_0}^{K-1} c_k q_k \le \frac{x}{1+\alpha^2} < (c_{k_0}+1) q_{k_0} + \sum_{k=k_0+1}^{K-1} c_k q_k . \]
The integer part thus also satisfies
\[ \sum_{k=k_0}^{K-1} c_k q_k \le \left\lfloor \frac{x}{1+\alpha^2} \right\rfloor <  (c_{k_0}+1) q_{k_0} + \sum_{k=k_0+1}^{K-1} c_k q_k , \]
showing that its Ostrowski expansion is of the form $\lfloor x/(1+\alpha^2) \rfloor =\sum_{k=0}^{k_0-1} b_k q_k +  \sum_{k=k_0}^{K-1} c_k q_k$, as claimed.

Assume next that $c_{k_0}=a$. By the construction of the $\alpha$-expansion, we now have
\[ \sum_{k=k_0}^{K-1} c_k \frac{\alpha^{-k}}{1+\alpha^2} \le \frac{x}{1+\alpha^2} < \frac{\alpha^{-k_0+1}}{1+\alpha^2} + \sum_{k=k_0}^{K-1} c_k \frac{\alpha^{-k}}{1+\alpha^2} . \]
We similarly deduce that
\[ \sum_{k=k_0}^{K-1} c_k q_k -\frac{1}{100} \le \frac{x}{1+\alpha^2} < q_{k_0-1} + \sum_{k=k_0}^{K-1} c_k q_k +\frac{1}{100}, \]
and by the assumption $x \not\in E$, that
\[ \sum_{k=k_0}^{K-1} c_k q_k \le \left\lfloor \frac{x}{1+\alpha^2} \right\rfloor < q_{k_0-1} + \sum_{k=k_0}^{K-1} c_k q_k . \]
Therefore the Ostrowski expansion is of the form $\lfloor x/(1+\alpha^2) \rfloor =\sum_{k=0}^{k_0-2} b_k q_k + \sum_{k=k_0}^{K-1} c_k q_k$, as claimed.
\end{proof}

The main advantage of the $\alpha$-expansion over the Ostrowski expansion is that the sequence of digits forms a Markov chain, which in turn leads to independence; a crucial property in the proof of the central limit theorem. More precisely, let us choose a real number $x$ from the interval $[0,\alpha^{-K})$ uniformly at random, and consider the $\alpha$-expansion $x = \sum_{k=-\infty}^{K-1} c_k \alpha^{-k}$. The digits $c_k$ are thus random variables. Using the geometric interpretation of the $\alpha$-expansion in terms of a partitioning system of nested intervals, it is not difficult to see that the sequence $c_k$ is in fact a Markov chain: the state space is the finite set $\{ 0,1,\ldots, a \}$, and the chain is in the state $c_{K-t}$ at time $t \in \mathbb{N}$. The distribution of the initial state, corresponding to the lengths of the partitioning intervals, is
\[ \alpha^{K} \lambda \left( \left\{ 0 \le x < \alpha^{-K} \, : \, c_{K-1}=c \right\} \right) = \left\{ \begin{array}{ll} \alpha & \textrm{if } 0 \le c <a, \\ \alpha^2 & \textrm{if } c=a . \end{array} \right. \]
The transition probabilities are described by the $(a+1)\times (a+1)$ matrix
\[ \left( \begin{array}{ccccc} \alpha & \alpha & \cdots & \alpha & \alpha^2 \\ \alpha & \alpha & \cdots & \alpha & \alpha^2 \\ \vdots & \vdots & \vdots & \vdots & \vdots \\ \alpha & \alpha & \cdots & \alpha & \alpha^2 \\ 1 & 0 & \cdots & 0 & 0 \end{array} \right) . \]
In other words, given any $0 \le c < a$, the probability of transitioning to a state $c'$ is $\alpha$ if $0 \le c' <a$, and $\alpha^2$ if $c'=a$. The last line of the matrix corresponds to the extra rule of $\alpha$-expansions: after the state $a$, we transition to the state $0$ with probability $1$. One readily checks that this Markov chain is irreducible and aperiodic, and that its (unique) stationary distribution vector is $(\frac{\alpha+\alpha^2}{1+\alpha^2}, \frac{\alpha}{1+\alpha^2}, \ldots, \frac{\alpha}{1+\alpha^2}, \frac{\alpha^2}{1+\alpha^2})$.

\subsubsection{The covariance matrix}

The covariance matrix of a complex-valued random variable $X=Y+iZ$ is defined as
\[ \mathrm{Cov} (X) = \left( \begin{array}{cc} \mathbb{E} (Y^2) & \mathbb{E} (YZ) \\ \mathbb{E} (YZ) & \mathbb{E} (Z^2) \end{array} \right) - \left( \begin{array}{cc} (\mathbb{E} Y)^2 & \mathbb{E} Y \mathbb{E} Z \\ \mathbb{E} Y \mathbb{E} Z & (\mathbb{E} Z)^2 \end{array} \right) . \]
Given two matrices $A,B$, the notation $A=B+O(f(K))$ means that $A_{mn}=B_{mn}+O(f(K))$ for all entries.

\begin{lem}\label{diophantinelemma2} For any $1 \le M' \le M$,
\[ \sum_{m=1}^M \frac{1}{m \| m \alpha \|} \min \left\{ \frac{1}{M' \| m \alpha \|}, 1 \right\} \ll \left( 1 + \log \frac{M}{M'} \right)^2 . \]
\end{lem}

\begin{proof} This is a slight modification of the third claim in Lemma \ref{diophantinelemma}. We have
\[ \sum_{q_{\ell} \le m < q_{\ell +1}} \frac{1}{m \| m \alpha \|} \cdot \frac{1}{M' \| m \alpha \|} \ll \frac{1}{q_{\ell} M'} \sum_{j=1}^{\infty} \frac{1}{j^2/q_{\ell +1}^2} \ll \frac{q_{\ell}}{M'} , \]
and by summing over all $\ell$ such that $q_{\ell} \le M'$, we obtain $\sum_{m=1}^{M'} \frac{1}{m \| m \alpha \|} \cdot \frac{1}{M' \| m \alpha \|} \ll 1$. There are $\ll 1+\log (M/M')$ convergent denominators $M'<q_{\ell} \le M$, and they all satisfy
\[ \sum_{q_{\ell} \le m < q_{\ell +1}} \frac{1}{m \| m \alpha \|} \min \left\{ \frac{1}{M' \| m \alpha \|} , 1 \right\} \ll \frac{1}{q_{\ell}} \sum_{j=1}^{\infty} \frac{1}{j/q_{\ell +1}} \min \left\{ \frac{1}{M' j/q_{\ell +1}} , 1 \right\} \ll \log \frac{q_{\ell +1}}{M'} . \]
Hence
\[ \sum_{m=M'+1}^M \frac{1}{m \| m \alpha \|} \min \left\{ \frac{1}{M' \| m \alpha \|} , 1 \right\} \ll \left( 1 + \log \frac{M}{M'} \right)^2 , \]
and the claim follows.
\end{proof}

\begin{lem}\label{covariancelemma} If the integer $N$ is chosen from $[0,q_K)$ uniformly at random, then $\mathbb{E} T_N = O((\log K)^2)$, and
\[ \mathrm{Cov} (T_N) = \left( \begin{array}{cc} \sigma(\alpha)^2 \log q_K & 0 \\ 0 & \sigma(\alpha)^2 \log q_K \end{array} \right) + O((\log K)^4) . \]
\end{lem}

\begin{proof} An application of Lemma \ref{diophantinelemma2} shows that the expected value is
\[ \begin{split} |\mathbb{E} T_N| &= \left| \frac{1}{q_K} \sum_{N=0}^{q_K-1} \sum_{1 \le m \le Q_K \log Q_K} \frac{e^{2 \pi i (N+\xi) m \alpha}}{2m \sin (\pi m \alpha)} \right| \ll \sum_{1 \le m \le Q_K \log Q_K} \frac{1}{m \| m \alpha \|} \min \left\{ \frac{1}{q_K \| m \alpha \|}, 1 \right\} \\ &\ll (\log K)^2, \end{split} \]
as claimed.

Repeating the steps in the estimate of $V_M$ in the proof of Proposition \ref{momentsprop} with obvious modifications, we deduce that $\mathbb{E} (\mathrm{Re} \, T_N)^2$ and $\mathbb{E} (\mathrm{Im} \, T_N)^2$ are both
\[ \sum_{m=1}^{q_K} \frac{1}{8 \pi^2 m^2 \| m \alpha \|^2} + O((\log K)^4) . \]
Formula \eqref{beckformula} shows that this is $\sigma (\alpha )^2 \log q_K + O((\log K)^4)$. Therefore the diagonal entries of $\mathrm{Cov}(T_N)$ are $\sigma (\alpha )^2 \log q_K + O((\log K)^4)$, as claimed. Next, consider
\[ \mathbb{E} (\mathrm{Re} \, T_N \mathrm{Im} \, T_N) = \frac{1}{q_K} \sum_{N=0}^{q_K-1} \sum_{1 \le m_1, m_2 \le Q_K \log Q_K} \frac{\sin (2 \pi (N+\xi)m_1 \alpha ) \cos (2 \pi (N+\xi)m_2 \alpha )}{4 m_1 m_2 \sin (\pi m_1 \alpha) \sin (\pi m_2 \alpha)} . \]
The contribution of the terms $m_1=m_2$ is
\[ \frac{1}{q_K} \sum_{N=0}^{q_K-1} \sum_{1 \le m \le Q_K \log Q_K} \frac{\sin (4 \pi (N+\xi) m \alpha)}{8m^2 \sin^2 (\pi m \alpha)} \ll \sum_{1 \le m \le Q_K \log Q_K} \frac{1}{m^2 \| m \alpha \|^2} \min \left\{ \frac{1}{q_K \| 2m \alpha \|}, 1 \right\} \ll \log K , \]
where we used the third claim in Lemma \ref{diophantinelemma}. The contribution of the terms $m_1 \neq m_2$ is
\[ \begin{split} \frac{1}{q_K} \sum_{N=0}^{q_K-1} \sum_{\substack{1 \le m_1, m_2 \le Q_K \log Q_K \\ m_1 \neq m_2}} &\frac{\sin (2 \pi (N+\xi)(m_1-m_2) \alpha) + \sin (2 \pi (N+\xi)(m_1+m_2) \alpha)}{8 m_1 m_2 \sin (\pi m_1 \alpha) \sin (\pi m_2 \alpha)} \\ \ll &\sum_{\substack{1 \le m_1, m_2 \le Q_K \log Q_K \\ m_1 \neq m_2}} \frac{1}{m_1 \| m_1 \alpha \| m_2 \| m_2 \alpha \|} \min \left\{ \frac{1}{q_K \| (m_1-m_2) \alpha \|}, 1 \right\} \\ &+\sum_{\substack{1 \le m_1, m_2 \le Q_K \log Q_K \\ m_1 \neq m_2}} \frac{1}{m_1 \| m_1 \alpha \| m_2 \| m_2 \alpha \|} \min \left\{ \frac{1}{q_K \| (m_1+m_2) \alpha \|}, 1 \right\} \\ \ll &(\log K)^4, \end{split} \]
where we used the last claim in Lemma \ref{diophantinelemma}. The off-diagonal entries of $\mathrm{Cov}(T_N)$ are thus $O((\log K)^4)$, as claimed.
\end{proof}

\subsubsection{Approximation in $L^2$}

Any $1 \le m \le Q_K \log Q_K$ has a unique Ostrowski expansion of the form $m=\sum_{k=h}^H d_k q_k$, where $0 \le h=h(m) \le H=H(m)$ are integers, and $d_h, d_H \neq 0$.
\begin{lem}\label{malphalemma} We have $\alpha^{h(m)-H(m)} \ll m \| m \alpha \| \ll \alpha^{h(m)-H(m)}$.
\end{lem}

\begin{proof} Note that $q_H \le m < q_{H+1}$, in particular, $\alpha^{-H} \ll m \ll \alpha^{-H}$. It remains to show that $\alpha^{h} \ll \| m \alpha \| \ll \alpha^{h}$. Observe that
\[ m \alpha = \sum_{k=h}^H d_k q_k \alpha \equiv \sum_{k=h}^H d_k (q_k \alpha - p_k) = \sum_{k=h}^H (-1)^k d_k \alpha^{k+1} \pmod{1} . \]
Since $0<d_h \le a$ and $0 \le d_k \le a$ for all $k$, we have
\[ d_h \alpha^{h+1} - d_{h+1} \alpha^{h+2} + d_{h+2} \alpha^{h+3} - \cdots \ge \alpha^{h+1} - a \alpha^{h+2} - a \alpha^{h+4} - \cdots = \alpha^{h+1} \left( 1 - \frac{a\alpha}{1+\alpha^2} \right) , \]
and similarly,
\[ d_h \alpha^{h+1} - d_{h+1} \alpha^{h+2} + d_{h+2} \alpha^{h+3} - \cdots \le a \alpha^{h+1} + a \alpha^{h+3} + a \alpha^{h+5} + \cdots = \frac{a \alpha^{h+1}}{1+\alpha^2} . \]
Therefore
\[ \alpha^h \ll \alpha^{h+1} \left( 1- \frac{a\alpha}{1+\alpha^2} \right) \le \| m \alpha \| \le \frac{a \alpha^{h+1}}{1+\alpha^2} \ll \alpha^h, \]
as claimed.
\end{proof}

For the rest of the proof, let
\[ k_0 :=  \left\lfloor \frac{10 \log K}{\log (1/\alpha)} \right\rfloor \textrm{ and } S := \left\{ 1 \le m \le Q_K \log Q_K \, : \, H(m)-h(m) \le k_0 \right\} . \]
Lemma \ref{malphalemma} yields the implication $m \not\in S \,\, \Rightarrow \,\, m \| m \alpha \| \gg K^{10}$. We start by showing that the terms in $T_N$ with $m \not\in S$ have a negligible contribution.
\begin{lem}\label{L2lemma1} We have
\[ \frac{1}{q_K} \sum_{N=0}^{q_K-1} \bigg| \sum_{\substack{1 \le m \le Q_K \log Q_K \\ m \not\in S}} \frac{e^{2 \pi i (N+\xi) m \alpha}}{2m \sin (\pi m \alpha)} \bigg|^2 \ll 1. \]
\end{lem}

\begin{proof} Expanding the square and considering the diagonal and off-diagonal terms separately leads to
\[ \begin{split}  &\frac{1}{q_K} \sum_{N=0}^{q_K-1} \bigg| \sum_{\substack{1 \le m \le Q_K \log Q_K \\ m \not\in S}} \frac{e^{2 \pi i (N+\xi) m \alpha}}{2m \sin (\pi m \alpha)} \bigg|^2 \\ &= \sum_{\substack{1 \le m \le Q_K \log Q_K \\ m \not\in S}} \frac{1}{4m^2 \sin^2 (\pi m \alpha)} + \frac{1}{q_K} \sum_{N=0}^{q_K-1} \sum_{\substack{1 \le m_1, m_2 \le Q_K \log Q_K \\ m_1, m_2 \not\in S \\ m_1 \neq m_2}} \frac{e^{2\pi i (N+\xi)(m_1-m_2) \alpha}}{4 m_1 m_2 \sin (\pi m_1 \alpha) \sin (\pi m_2 \alpha)} \\ &\ll \sum_{\substack{1 \le m \le Q_K \log Q_K \\ m \not\in S}} \frac{1}{m^2 \| m \alpha \|^2} + \sum_{\substack{1 \le m_1, m_2 \le Q_K \log Q_K \\ m_1, m_2 \not\in S \\ m_1 \neq m_2}} \frac{1}{m_1 \| m_1 \alpha \| m_2 \| m_2 \alpha \|} \min \left\{ \frac{1}{q_K \| (m_1-m_2) \alpha \|} ,1 \right\} \\ &\ll \frac{1}{K^{10}} \sum_{1 \le m \le Q_K \log Q_K} \frac{1}{m \| m \alpha \|} + \frac{1}{K^{10}} \sum_{1 \le m_1 \le Q_K \log Q_K} \frac{1}{m_1 \| m_1 \alpha \|} \sum_{1 \le m' \le Q_K \log Q_K} \frac{1}{q_K \| m' \alpha \|} \\ &\ll 1. \end{split} \]
Note that instead of $m_1$ and $m_2$, we summed over $m_1$ and $m' := |m_1-m_2|$, and used the first claim in Lemma \ref{diophantinelemma} combined with summation by parts.
\end{proof}

We now show that the dependence of $T_N$ on the first $k_0$ Ostrowski digits of $N$ is very weak.
\begin{lem}\label{L2lemma2} For any integer $0 \le N <q_K$ with Ostrowski expansion $N=\sum_{k=0}^{K-1} b_k q_k$, let $g(N)=\sum_{k=k_0}^{K-1} b_k q_k$. We have
\[ T_N = \sum_{m \in S} \frac{e^{2 \pi i (g(N)+\xi ) m \alpha}}{2m \sin (\pi m \alpha)} + Z_N , \qquad 0 \le N < q_K \]
with some $Z_N$ that satisfies $q_K^{-1} \sum_{N=0}^{q_K-1} |Z_N|^2 \ll (\log K)^4$.
\end{lem}

\begin{proof} The error term $Z_N$ in the lemma can be written in the form
\[ \begin{split} Z_N := &T_N -  \sum_{m \in S} \frac{e^{2 \pi i (g(N)+\xi) m \alpha}}{2m \sin (\pi m \alpha)} \\ = &\sum_{\substack{1 \le m \le Q_K \log Q_K \\ m \not\in S}} \frac{e^{2 \pi i (N+\xi) m \alpha}}{2m \sin (\pi m \alpha)} + \sum_{m \in S} \frac{(e^{2 \pi i r(N) m \alpha}-1)e^{2 \pi i (g(N)+\xi) m \alpha}}{2m \sin (\pi m \alpha)} , \end{split} \]
where $r(N):=N-g(N)$. By the triangle inequality for the $L^2$ norm and Lemma \ref{L2lemma1}, it will be enough to estimate the second moment of the second term, that is,
\[ \frac{1}{q_K} \sum_{N=0}^{q_K-1} \bigg| \sum_{m \in S} \frac{(e^{2 \pi i r(N) m \alpha}-1)e^{2\pi i (g(N)+\xi ) m \alpha}}{2m \sin (\pi m \alpha)} \bigg|^2 = R_{\mathrm{diag}}+R_{\mathrm{off-diag}}, \]
where
\[ \begin{split} R_{\mathrm{diag}} &= \frac{1}{q_K} \sum_{N=0}^{q_K-1} \sum_{m \in S} \frac{|e^{2 \pi i r(N)m \alpha}-1|^2}{4 m^2 \sin^2 (\pi m \alpha)}, \\ R_{\mathrm{off-diag}} &= \frac{1}{q_K} \sum_{N=0}^{q_K-1} \sum_{\substack{m_1, m_2 \in S \\ m_1 \neq m_2}} \frac{(e^{2 \pi i r(N)m_1 \alpha}-1)(e^{-2 \pi i r(N)m_2 \alpha}-1) e^{2 \pi i (g(N)+\xi ) (m_1-m_2) \alpha} }{4 m_1 m_2 \sin (\pi m_1 \alpha) \sin (\pi m_2 \alpha)} . \end{split} \]
We thus need to show that $R_{\mathrm{diag}} + R_{\mathrm{off-diag}} \ll (\log K)^4$. Clearly, $r(N) = \sum_{k=0}^{k_0-1} b_k q_k < q_{k_0} \ll K^{10}$. Therefore
\[ \sum_{m \in S} \frac{|e^{2 \pi i r(N)m \alpha}-1|^2}{4 m^2 \sin^2 (\pi m \alpha)} \ll \sum_{1 \le m \le K^{20}} \frac{1}{m^2 \| m \alpha \|^2} + \sum_{K^{20} < m \le Q_K \log Q_K} \frac{r(N)^2}{m^2} \ll \log K \]
uniformly in $0 \le N < q_K$, where we used formula \eqref{m2malpha2sum}. This shows that $R_{\mathrm{diag}} \ll \log K$.

We now estimate $R_{\mathrm{off-diag}}$. Let $A(N_0)= \{ 0 \le N < q_K \, : \, r(N)=N_0 \}$. Note that $g(N)=N-N_0$ for all $N \in A(N_0)$. For the sake of readability, set
\[ W(N_0,m_1,m_2):= (e^{2 \pi i N_0 m_1 \alpha}-1)(e^{-2 \pi i N_0 m_2 \alpha}-1) e^{2 \pi i (-N_0+\xi) (m_1-m_2) \alpha}. \]
Then $R_{\mathrm{off-diag}}$ can be written in the form
\[ R_{\mathrm{off-diag}} = \frac{1}{q_K} \sum_{0 \le N_0 \ll K^{10}} \sum_{\substack{m_1, m_2 \in S \\ m_1 \neq m_2}} \frac{W(N_0,m_1,m_2)}{4m_1 m_2 \sin (\pi m_1 \alpha) \sin (\pi m_2 \alpha)} \sum_{N \in A(N_0)} e^{2 \pi i N (m_1-m_2) \alpha} . \]
Clearly $|W(N_0,m_1,m_2)| \le 4$, and also $|W(N_0,m_1,m_2)| \ll \| m_1 \alpha \| \| m_2 \alpha \| N_0^2$. Let
\[ D(N_0)= \bigg\{ 1 \le d \le Q_K \log Q_K \, : \, \bigg| \sum_{N \in A(N_0)} e^{2 \pi i N d \alpha} \bigg| \ge \frac{q_K}{K^{14}} \bigg\} . \]
The contribution of the terms for which $|m_1-m_2| \not\in D(N_0)$ is negligible:
\begin{multline*}
\bigg| \sum_{\substack{m_1, m_2 \in S \\ 0<|m_1-m_2| \not\in D(N_0)}} \frac{W(N_0,m_1,m_2)}{4m_1 m_2 \sin (\pi m_1 \alpha) \sin (\pi m_2 \alpha)} \sum_{N \in A(N_0)} e^{2 \pi i N (m_1-m_2) \alpha} \bigg| \\ \ll \sum_{1 \le m_1, m_2 \le Q_K \log Q_K} \frac{1}{m_1 \| m_1 \alpha \| m_2 \| m_2 \alpha \|} \cdot \frac{q_K}{K^{14}} \ll \frac{q_K}{K^{10}} .
\end{multline*}
The contribution of the terms for which $|m_1-m_2|=d \in D(N_0)$ is
\begin{multline*}
\bigg| \sum_{\substack{m_1, m_2 \in S \\ |m_1-m_2|=d}} \frac{W(N_0,m_1,m_2)}{4m_1 m_2 \sin (\pi m_1 \alpha) \sin (\pi m_2 \alpha)} \sum_{N \in A(N_0)} e^{2 \pi i N (m_1-m_2) \alpha} \bigg| \\ \ll |A(N_0)| \sum_{\substack{1 \le m_1, m_2 \le Q_K \log Q_K \\ |m_1-m_2|=d}} \frac{|W(N_0,m_1,m_2)|}{m_1 \| m_1 \alpha \| m_2 \| m_2 \alpha \|} .
\end{multline*}
By symmetry it is enough to sum over $m_1<m_2$, in which case $m_2=m_1+d$. Hence
\[ R_{\mathrm{off-diag}} \ll \frac{1}{q_K} \sum_{0 \le N_0 \ll K^{10}} \left( \frac{q_K}{K^{10}} + |A(N_0)| \sum_{d \in D(N_0)} \sum_{1 \le m \le Q_K \log Q_K} \frac{|W(N_0,m,m+d)|}{m \| m \alpha \| (m+d) \| (m+d) \alpha \|} \right) . \]
Here
\[ \sum_{1 \le m \le Q_K \log Q_K} \frac{|W(N_0,m,m+d)|}{m \| m \alpha \| (m+d) \| (m+d) \alpha \|} \ll \sum_{m=1}^{\infty} \frac{N_0^2}{m (m+d)} \ll K^{20} \frac{\log d}{d}, \]
which is negligible if, say, $d > K^{50}$. For smaller $d$, the contribution of all $m > K^{50}$ is also negligible:
\[ \sum_{K^{50} < m \le Q_K \log Q_K} \frac{|W(N_0,m,m+d)|}{m \| m \alpha \| (m+d) \| (m+d) \alpha \|} \ll \sum_{K^{50} < m \le Q_K \log Q_K} \frac{N_0^2}{m (m+d)} \ll \frac{1}{K^{30}} . \]
Therefore our estimate for $R_{\mathrm{off-diag}}$ simplifies to
\[ \begin{split} R_{\mathrm{off-diag}} &\ll 1 + \frac{1}{q_K} \sum_{0 \le N_0 \ll K^{10}} |A(N_0)| \left( \frac{\log K}{K^{30}} |D(N_0)| + \sum_{1 \le m, d \le K^{50}} \frac{1}{m \| m \alpha \| (m+d) \| (m+d) \alpha \|} \right) \\ &\ll 1+\frac{1}{q_K} \sum_{0 \le N_0 \ll K^{10}} |A(N_0)| \left( \frac{\log K}{K^{30}} |D(N_0)| + (\log K)^4 \right) \\ &\ll \frac{\log K}{K^{30}} \max_{0 \le N_0 \ll K^{10}} |D(N_0)| + (\log K)^4. \end{split} \]
In the last step we used $\sum_{N_0 \ge 0} |A(N_0)|=q_K$. It remains to estimate $|D(N_0)|$. For any $N_0 \ge 0$,
\[ \begin{split} \sum_{1 \le d \le Q_K \log Q_K} \bigg| \sum_{N \in A(N_0)} e^{2 \pi i N d \alpha} \bigg|^2 &= \sum_{1 \le d \le Q_K \log Q_K} \bigg( |A(N_0)| + \sum_{\substack{N, N' \in A(N_0) \\ N \neq N'}} e^{2 \pi i (N-N')d \alpha} \bigg) \\ &\ll |A(N_0)| q_K \log q_K + \sum_{\substack{N,N' \in A(N_0) \\ N \neq N'}} \frac{1}{\| (N-N') \alpha \|} \\ &\ll q_K^2 K + q_K \sum_{\ell =1}^{q_K-1} \frac{1}{\| \ell \alpha \|} \\ &\ll q_K^2 K. \end{split} \]
An application of the Chebyshev inequality thus yields $|D(N_0)| \ll K^{29}$. Therefore $R_{\mathrm{off-diag}} \ll (\log K)^4$, as claimed.
\end{proof}

\subsubsection{Approximation by a Markov chain}\label{markovchainsection}

The main goal of this section is to approximate $T_N$ in terms of the Markov chain formed by the digits in the $\alpha$-expansion of a real number. Throughout, the integer $N$ is chosen from $[0,q_K)$ uniformly at random, whereas the real $x$ is chosen from $[0,\alpha^{-K})$ uniformly at random and has $\alpha$-expansion $x=\sum_{k=-\infty}^{K-1} c_k \alpha^{-k}$.

The Ostrowski expansion of any $1 \le m \le Q_K \log Q_K$ is of the form $m=\sum_{k=h}^H d_k q_k$ with some integers $0 \le h=h(m) \le H=H(m)$ and some digits $d_k=d_k(m)$, where $d_h, d_H \neq 0$. We use the convention $d_k=d_k(m)=0$ whenever $k \not\in [h,H]$. Let
\[ G(x,m) = \frac{\alpha}{1+\alpha^2} \left( \sum_{k=k_0}^{K-1} (-1)^k c_k d_k + \sum_{j=1}^{k_0} \alpha^j \sum_{k=k_0}^{K-1} (-1)^{k+j} c_k d_{k+j} + \sum_{j=1}^{k_0} \alpha^j \sum_{k=k_0}^{K-1} (-1)^k c_k d_{k-j} \right) , \]
and define
\[ U(x) = \sum_{m \in S} \frac{e^{2 \pi i (G(x,m) + \xi m \alpha )}}{2 m \sin (\pi m \alpha)} . \]
The point is that for a given $m \in S$, the function $G(x,m)$ depends only on $c_k$, $k \in [h-k_0,H+k_0]$, that is, only on $\ll k_0 \ll \log K$ of all the variables $c_k$, $k_0 \le k \le K-1$.

For the sake of readability, from now on we use fully probabilistic notation. We thus write
\[ \begin{split} \Pr \left( T_N \in A \right) &= \frac{1}{q_K} \left| \left\{ 0 \le N < q_K \, : \, T_N \in A \right\} \right| , \\ \Pr \left( U(x) \in A \right) &= \alpha^K \lambda \left( \left\{ 0 \le x < \alpha^{-K} \, : \, U(x) \in A \right\} \right) \end{split} \]
for a Lebesgue measurable set $A \subseteq \mathbb{C}$. Further, let
\[ \Sigma_K = \left( \begin{array}{cc} \sigma (\alpha )^2 \log q_K & 0 \\ 0 & \sigma (\alpha )^2 \log q_K \end{array} \right) . \]
The following lemma reduces the central limit theorem for $T_N$ to that for $U(x)$.
\begin{lem}\label{UxCLTlemma} We have $\mathbb{E} U(x)=O((\log K)^2)$ and $\mathrm{Cov} (U(x)) = \Sigma_K + O(K^{1/2}(\log K)^2)$. Further,
\[ \sup_{C \subseteq \mathbb{C}} \left| \Pr \left( \frac{T_N}{\sigma(\alpha) \sqrt{\log q_K}} \in C \right) - \Phi (C) \right| = \sup_{C \subseteq \mathbb{C}} \left| \Pr \left( \frac{U(x)}{\sigma (\alpha) \sqrt{\log q_K}} \in C \right) - \Phi (C) \right| + O \left( \frac{(\log K)^{4/3}}{K^{1/3}} \right) , \]
where the supremum is over all convex sets $C \subseteq \mathbb{C}$, and $\Phi (C) = \int_{C} \frac{e^{-|y|^2 /2}}{2 \pi} \, \mathrm{d}y$.
\end{lem}

\begin{proof} Let $N_x= \lfloor x/(1+\alpha^2) \rfloor$. Since $0 \le x/(1+\alpha^2) < \alpha^{-K}/(1+\alpha^2) = q_K +O(\alpha^K)$, the random variable $N_x$ is supported on either $[0,q_K)$ or $[0,q_K]$. For any $0 \le n \le q_K-2$ we have
\[ \Pr (N_x=n)=\alpha^K (1+\alpha^2) = q_K^{-1}+O(\alpha^{3K})=\Pr (N=n)+O(\alpha^{3K}), \]
and for $n \in \{ q_K-1, q_K \}$ we still have $\Pr (N_x=n)=\Pr (N=n)+O(\alpha^{2K})$. Therefore the distributions of $N_x$ and $N$ have exponentially small distance in total variation:
\begin{equation}\label{TV}
\sum_{n=0}^{\infty} |\Pr (N_x=n) - \Pr (N=n)| \ll \alpha^{2K} .
\end{equation}
Using the pointwise bound
\begin{equation}\label{pointwisebound}
|T_N| \ll \sum_{1 \le m \le Q_K \log Q_K} \frac{1}{m \| m \alpha \|} \ll (\log Q_K)^2 \ll K^2
\end{equation}
and the estimate \eqref{TV}, we deduce $\mathbb{E} T_{N_x} = \mathbb{E} T_N+O(K^2 \alpha^{2K})$ and $\mathrm{Cov}(T_{N_x}) = \mathrm{Cov} (T_N) +O(K^4 \alpha^{2K})$. Hence by Lemma \ref{covariancelemma}, $\mathbb{E} T_{N_x}=O((\log K)^2)$ and $\mathrm{Cov} (T_{N_x}) = \Sigma_K +O((\log K)^4)$.

Recall the definition of $g$ and $Z_N$ from Lemma \ref{L2lemma2}. We similarly deduce $\mathbb{E} |Z_{N_x}|^2 = \mathbb{E} |Z_N|^2 + O(K^4 \alpha^{2K})=O((\log K)^4)$. Therefore
\[ U_1(x) := \sum_{m \in S} \frac{\exp \left( 2 \pi i \left( g(N_x)+\xi \right) m \alpha \right)}{2 m \sin (\pi m \alpha)} = T_{N_x}-Z_{N_x} \]
satisfies $\mathbb{E} U_1(x)=O((\log K)^2)$ and $\mathrm{Cov} (U_1(x)) = \Sigma_K + O(K^{1/2} (\log K)^2)$. Using the coupling between the Ostrowski expansion and the $\alpha$-expansion constructed in Lemma \ref{ostrowskilemma}, here $g(N_x)=\sum_{k=k_0}^{K-1} c_k q_k$ with probability $\ge 1-O(K^{-10})$. This together with a pointwise bound of the form \eqref{pointwisebound} yields that
\[ U_2(x) := \sum_{m \in S} \frac{\exp \left( 2 \pi i \left( \sum_{k=k_0}^{K-1} c_k q_k + \xi \right) m \alpha \right)}{2 m \sin (\pi m \alpha)} \]
also satisfies $\mathbb{E}U_2(x)=O((\log K)^2)$ and $\mathrm{Cov} (U_2(x)) = \Sigma_K +O(K^{1/2} (\log K)^2)$.

Observe that
\[ \sum_{k=k_0}^{K-1} c_k q_k m \alpha = \sum_{k=k_0}^{K-1} \sum_{\ell=0}^{\infty} c_k q_k d_{\ell} q_{\ell} \alpha = \sum_{k=k_0}^{K-1} c_k \left( d_k q_k^2 \alpha + \sum_{j=1}^{\infty} d_{k+j} q_k q_{k+j} \alpha + \sum_{j=1}^{k} d_{k-j} q_k q_{k-j} \alpha \right) . \]
Working modulo $1$, here
\[ \begin{split} q_k^2 \alpha &\equiv q_k (-1)^k \alpha^{k+1} = \frac{\alpha}{1+\alpha^2}(-1)^k +O(\alpha^{2k}) \pmod{1},  \\ q_k q_{k+j} \alpha &\equiv q_k (-1)^{k+j} \alpha^{k+j+1} = \frac{\alpha}{1+\alpha^2} (-1)^{k+j} \alpha^j + O(\alpha^{2k+j}) \pmod{1}, \\ q_k q_{k-j} \alpha & \equiv q_{k-j} (-1)^k \alpha^{k+1} = \frac{\alpha}{1+\alpha^2} (-1)^k \alpha^j + O(\alpha^{2k-j}) \pmod{1}, \end{split} \]
hence
\[ \sum_{k=k_0}^{K-1} c_k q_k m \alpha \equiv \frac{\alpha}{1+\alpha^2} \sum_{k=k_0}^{K-1} (-1)^k c_k \left( d_k + \sum_{j=1}^{\infty} (- \alpha)^j d_{k+j} + \sum_{j=1}^{k} \alpha^j d_{k-j} \right) +O(\alpha^{k_0}) \pmod{1} . \]
Discarding the terms $j>k_0$ from both sums in the previous formula introduces the error $\sum_{k=k_0}^{K-1} \sum_{j=k_0+1}^{\infty} \alpha^j \ll K \alpha^{k_0} \ll K^{-9}$, thus $\sum_{k=k_0}^{K-1} c_k q_k m \alpha \equiv G(x,m) + O(K^{-9}) \pmod{1}$. In particular,
\begin{equation}\label{U2-U}
|U_2(x)-U(x)| \ll \sum_{1 \le m \le Q_K \log Q_K} \frac{K^{-9}}{m \| m \alpha \|} \ll K^{-7}.
\end{equation}
Therefore $\mathbb{E} U(x) = O((\log K)^2)$ and $\mathrm{Cov} (U(x)) = \Sigma_K + O(K^{1/2} (\log K)^2)$, as claimed.

An application of the Chebyshev inequality yields
\[ \Pr \left( \frac{|Z_{N_x}|}{\sigma (\alpha) \sqrt{\log q_K}} \ge \frac{(\log K)^{4/3}}{K^{1/3}} \right) \ll \frac{(\log K)^{4/3}}{K^{1/3}} . \]
Formulas \eqref{TV} and \eqref{U2-U} together with $\Pr (U_1(x) \neq U_2(x)) \ll K^{-10}$ show that for any of the random variables $Y \in \{ T_N, T_{N_x}, U_1(x), U_2(x), U(x) \}$,
\[ \sup_{C \subseteq \mathbb{C}} \left| \Pr \left( \frac{Y}{\sigma (\alpha) \sqrt{\log q_K}} \in C \right) - \Phi(C) \right| \]
is the same up to an error $O(K^{-1/3} (\log K)^{4/3})$. Note that we tacitly used the fact that with the notation $C^{\varepsilon}:= \{ y \in \mathbb{R}^2 \, : \, \mathrm{dist} (y, \partial C) \le \varepsilon \}$, we have $\Phi (C^{\varepsilon}) \ll \varepsilon$ uniformly in convex sets $C \subseteq \mathbb{C}$ \cite[p.\ 24]{BRR}. This finishes the proof of the second claim of the lemma.
\end{proof}

From now on we can thus work with the random variable $U(x)$, which is expressed in terms of the Markov chain $c_{K-1}, c_{K-2}, \ldots$, discussed in Section \ref{ostrowskisection}. By classical results on finite state space irreducible, aperiodic Markov chains, $c_{K-j}$ converges to the stationary distribution exponentially fast as $j \to \infty$. It is also not difficult to see that the sequence $c_{K-1}, c_{K-2}, \ldots$ is $\psi$-mixing with exponential rate. That is, for any integers $j,\ell \ge 1$, letting $\mathcal{B}_j$ denote the $\sigma$-algebra generated by $c_{K-1}, c_{K-2}, \ldots, c_{K-j}$ and $\mathcal{F}_{j+\ell}$ the $\sigma$-algebra generated by $c_{K-j-\ell}, c_{K-j-\ell-1},\ldots$, the $\psi$-mixing coefficients satisfy
\begin{equation}\label{psimixing}
\psi(\ell) := \sup_{j \ge 1} \sup_{\substack{B \in \mathcal{B}_j, \,\,\, \Pr (B)>0 \\ F \in \mathcal{F}_{j+\ell}, \,\,\, \Pr (F)>0}} \bigg| \frac{\Pr (B \cap F)}{\Pr(B) \Pr (F)} -1 \bigg| \ll e^{-\tau \ell}
\end{equation}
with some constant $\tau >0$. In particular (cf.\ $\rho$-mixing), if $f(x)$ is a $\mathcal{B}_j$-measurable and $g(x)$ is an $\mathcal{F}_{j+\ell}$-measurable complex-valued function, then
\begin{equation}\label{rhomixing}
\mathbb{E} (f(x)g(x)) = \mathbb{E} f(x) \mathbb{E} g(x)+ O \left( \left( \mathbb{E} |f(x)|^2 \right)^{1/2} \left( \mathbb{E} |g(x)|^2 \right)^{1/2}  e^{-\tau \ell} \right) .
\end{equation}
We refer to \cite{BR} for a survey on mixing properties of Markov chains.

Since $\psi$-mixing is usually stated only for stationary Markov chains, but our chain $c_{K-1}, c_{K-2}, \ldots$ is not started from the stationary distribution, we include a short proof of \eqref{psimixing}. Fix $j, \ell \ge 1$, and assume first that $B \in \mathcal{B}_j$ and $F \in \mathcal{F}_{j+\ell}$ are the events
\[ \begin{split} B&=\{ c_{K-1} = x_{K-1}, \,\, c_{K-2} = x_{K-2}, \ldots, c_{K-j}=x_{K-j} \} , \\ F&= \{ c_{K-j-\ell} = x_{K-j-\ell}, \,\, c_{K-j-\ell -1} = x_{K-j-\ell -1}, \ldots \} \end{split} \]
with some $x_{K-1}, x_{K-2}, \ldots \in \{ 0,1,\ldots, a \}$. Using the Markov property,
\[ \begin{split} \Pr (B \cap F) &= \Pr (F \mid c_{K-j} = x_{K-j}) \Pr (B) \\ &= \Pr (F \mid c_{K-j-\ell}=x_{K-j-\ell}) \Pr (c_{K-j-\ell} = x_{K-j-\ell} \mid c_{K-j} = x_{K-j}) \Pr (B) \\ &= \frac{\Pr (F)}{\Pr (c_{K-j-\ell} = x_{K-j-\ell})} \Pr (c_{K-j-\ell} = x_{K-j-\ell} \mid c_{K-j} = x_{K-j}) \Pr (B) . \end{split} \]
Hence
\[ \left| \frac{\Pr (B \cap F)}{\Pr (B) \Pr (F)} -1 \right| = \left| \frac{\Pr (c_{K-j-\ell} = x_{K-j-\ell} \mid c_{K-j} = x_{K-j})}{\Pr (c_{K-j-\ell} = x_{K-j-\ell})} -1  \right| . \]
Here $\Pr (c_{K-j-\ell} = x_{K-j-\ell} \mid c_{K-j} = x_{K-j})$ and $\Pr (c_{K-j-\ell} = x_{K-j-\ell})$ are both exponentially close in $\ell$ to the stationary distribution at $x_{K-j-\ell}$, thus
\[ |\Pr (B \cap F) - \Pr (B) \Pr (F)| \ll e^{-\tau \ell} \Pr (B) \Pr (F) \]
uniformly in $j \ge 1$, with some constant $\tau >0$. The same holds for any $B \in \mathcal{B}_j$ and $F \in \mathcal{F}_{j+\ell}$ by $\sigma$-additivity. This finishes the proof of \eqref{psimixing}.

\subsubsection{Approximation by independent variables}\label{independentsection}

In this section we approximate $U(x)$ by a sum of independent random variables, and thus establish \eqref{TNCLT}.

Let $1 \le R \le K^{0.99}$ be an integer, to be chosen, and let us divide $[0,K]$ into $R$ subintervals $I_j=[(j-1)K/R,jK/R]$, $1 \le j \le R$ of equal length $K/R \ge K^{1/100}$. Set also $I_0=(-\infty,0]$ and $I_{R+1}=[K,\infty)$. Let $b>1$ be a large constant, to be chosen, and for any $1 \le j \le R+1$, define
\[ \begin{split} S_j&= \left\{ m \in S \, : \, h(m)- b k_0 \in I_j, \,\,\, H(m)+ b k_0 \in I_j \right\}, \\ S_j'&= \left\{ m \in S \, : \, h(m)-bk_0 \in I_{j-1}, \,\,\, H(m)+b k_0 \in I_j \right\} . \end{split} \]
Note that the length of the interval $[h(m)-bk_0,H(m)+bk_0]$ is negligible compared to $K/R$, therefore the sets $S_j,S_j'$ partition $S$. Since $h(m) \le K+O(\log K)$ for all $1 \le m \le Q_K \log Q_K$, we actually have $S_{R+1}=\emptyset$ provided that $b>1$ is chosen large enough. Consequently, we can express $U(x)$ as $U(x)=\sum_{j=1}^R X_j(x) + \sum_{j=1}^{R+1} Y_j(x)$, and in particular,
\[ U(x)-\mathbb{E} U(x) = \sum_{j=1}^R (X_j (x) - \mathbb{E} X_j (x)) + \sum_{j=1}^{R+1} (Y_j (x) - \mathbb{E} Y_j (x)) , \]
with the random variables
\[ X_j(x) := \sum_{m \in S_j} \frac{e^{2 \pi i (G(x,m)+\xi m \alpha)}}{2m \sin (\pi m \alpha)} \qquad \textrm{and} \qquad Y_j(x) := \sum_{m \in S_j'} \frac{e^{2 \pi i (G(x,m)+ \xi m \alpha)}}{2m \sin (\pi m \alpha)} . \]
We first show that for all $j$,
\begin{equation}\label{XjYjbound}
|X_j(x)| \ll \frac{K \log K}{R} \qquad \textrm{and} \qquad |Y_j(x)| \ll (\log K)^2 .
\end{equation}
Observe that for any given $0 \le h \le H$,
\[ \bigg| \sum_{\substack{m \in S\\ h(m)=h, \,\,\, H(m)=H}} \frac{e^{2 \pi i (G(x,m)+\xi m \alpha)}}{2m \sin (\pi m \alpha)} \bigg| \ll \sum_{\substack{m \in S\\ h(m)=h, \,\,\, H(m)=H}} \frac{1}{m \| m \alpha \|} \ll 1. \]
Indeed, by Lemma \ref{malphalemma} here each term is $\ll \alpha^{H-h}$, and the number of terms is equal to the number of legitimate Ostrowski sequences $(d_h,d_{h+1}, \ldots, d_H)$, which is $\ll q_{H-h} \ll \alpha^{h-H}$. Summing over $h,H$ in the appropriate range leads to \eqref{XjYjbound}.

The main observation is that $X_j(x)$ depends only on $c_k$, $\min I_j+(b-1)k_0 \le k \le \max I_j - (b-1)k_0$. Similarly, $Y_j(x)$ depends only on $c_k$, $\min I_j -(b+2)k_0 \le k \le \min I_j +(b+2)k_0$.

For any $1 \le j < j' \le R+1$, the functions $Y_j (x)$ resp.\ $Y_{j'}(x)$ depend on sets of variables $c_k$ with the indices $k$ separated by at least $K/R - 2(b+2)k_0 \ge K^{1/100}/2$. By \eqref{rhomixing} and \eqref{XjYjbound}, we have
\[ \mathbb{E} \left( Y_j (x) - \mathbb{E}Y_j (x) \right) \left( \overline{Y_{j'} (x)} - \mathbb{E}\overline{Y_{j'} (x)} \right) \ll (\log K)^4 e^{-\tau K^{1/100}/2} \ll K^{-2} . \]
Hence
\begin{equation}\label{Yjvariancebound}
\mathbb{E} \left| \sum_{j=1}^{R+1} (Y_j (x) - \mathbb{E} Y_j (x)) \right|^2 \ll \sum_{j=1}^{R+1} \mathbb{E} |Y_j (x) - \mathbb{E} Y_j (x)|^2  +1 \ll R (\log K)^4 ,
\end{equation}
and by Lemma \ref{UxCLTlemma},
\begin{equation}\label{Xjvariance}
\begin{split} \mathrm{Cov} \left( \sum_{j=1}^R (X_j (x) - \mathbb{E} X_j (x)) \right) &= \mathrm{Cov} (U(x)) + O \left( K^{1/2} R^{1/2} (\log K)^2 \right) \\ &= \Sigma_K + O \left( K^{1/2} R^{1/2} (\log K)^2 \right) . \end{split}
\end{equation}

The main contribution to $U(x)$ is the sum of $X_j(x)$. For any $2 \le j \le R$, the vector $(X_1 (x), X_2 (x), \ldots, X_{j-1} (x))$ resp.\ the variable $X_j (x)$ depend on sets of variables $c_k$ with the indices $k$ separated by at least $(2b-2) k_0$. By an approximation theorem of Berkes and Philipp \cite[Theorem 2]{BP}, on a suitable probability space there exist complex-valued random variables $X_j^*$ and $\zeta_j$, $1 \le j \le R$, with $\zeta_j$, $1 \le j \le R$ independent, such that the vectors $(X_1 (x), X_2(x), \ldots, X_R (x))$ and $(X_1^*, X_2^*, \ldots, X_R^*)$ are identically distributed, $X_j^*$ and $\zeta_j$ are also identically distributed for any given $1 \le j \le R$, and $\Pr (|X_j^* - \zeta_j| \ge 6 \psi ((2b-2)k_0)) \le 6 \psi ((2b-2)k_0)$, where $\psi$ denotes the $\psi$-mixing coefficients from \eqref{psimixing}. Here $\psi ((2b-2)k_0) \ll K^{-100}$ provided that $b>1$ is chosen large enough, thus
\[ \Pr \left( \sum_{j=1}^R |X_j^* - \zeta_j| \ge K^{-99} \right) \ll K^{-99} . \]
By \eqref{XjYjbound}, we have $|X_j^*|, |\zeta_j| \ll (K \log K) /R$, and the previous formula yields in particular that $\mathbb{E} |\sum_{j=1}^R (X_j^*-\zeta_j)|^2 \ll K^{-96}$. Formula \eqref{Xjvariance} thus leads to
\[ \mathrm{Cov} \left( \sum_{j=1}^R (\zeta_j - \mathbb{E} \zeta_j) \right) = \Sigma_K + O(K^{1/2} R^{1/2} (\log K)^2) . \]
The random variables $(\zeta_j-\mathbb{E}\zeta_j)/(\sigma (\alpha) \sqrt{\log q_K})$, $1 \le j \le R$ are thus independent, mean zero, and
\[ \Sigma_{*} := \mathrm{Cov} \left( \frac{\sum_{j=1}^R (\zeta_j - \mathbb{E} \zeta_j)}{\sigma (\alpha) \sqrt{\log q_K}} \right) = \left( \begin{array}{cc} 1 & 0 \\ 0 & 1 \end{array} \right) + O \left( K^{-1/2} R^{1/2} (\log K)^2 \right) . \]
Further,
\[ \sum_{j=1}^R \mathbb{E} \left| \frac{\zeta_j - \mathbb{E} \zeta_j}{\sigma (\alpha) \sqrt{\log q_K}} \right|^3 \ll \frac{K^{1/2} \log K}{R} \sum_{j=1}^R \mathbb{E} \left| \frac{\zeta_j - \mathbb{E} \zeta_j}{\sigma (\alpha) \sqrt{\log q_K}} \right|^2 \ll \frac{K^{1/2} \log K}{R} . \]
The multidimensional central limit theorem with explicit rate \cite[p.\ 166]{BRR} thus gives
\[ \sup_{C \subseteq \mathbb{C}} \left| \Pr \left( \frac{\sum_{j=1}^R (\zeta_j - \mathbb{E} \zeta_j)}{\sigma (\alpha) \sqrt{\log q_K}} \in C \right) - \Phi_{\Sigma_{*}}(C) \right| \ll \sum_{j=1}^R \mathbb{E} \left| \frac{\zeta_j - \mathbb{E} \zeta_j}{\sigma (\alpha) \sqrt{\log q_K}} \right|^3 \ll \frac{K^{1/2} \log K}{R} , \]
where $\Phi_{\Sigma_{*}}(C)=\int_C e^{-\langle y,\Sigma_{*}^{-1} y \rangle/2}/(2 \pi \det (\Sigma_{*})^{1/2}) \, \mathrm{d} y$ is the mean zero Gaussian with covariance matrix $\Sigma_{*}$. It is easy to see that $\Phi_{\Sigma_{*}}(C) = \Phi (C) +O(K^{-1/2}R^{1/2} (\log K)^2)$ uniformly in convex sets $C \subseteq \mathbb{C}$, hence
\[ \sup_{C \subseteq \mathbb{C}} \left| \Pr \left( \frac{\sum_{j=1}^R (X_j (x) - \mathbb{E} X_j (x))}{\sigma (\alpha) \sqrt{\log q_K}} \in C \right) - \Phi (C) \right| \ll \frac{K^{1/2} \log K}{R} + \frac{R^{1/2} (\log K)^2}{K^{1/2}} . \]

Similarly, $Y_j (x)$, $1 \le j \le R+1$ can be extremely well approximated by independent random variables $\zeta_j'$, $1 \le j \le R+1$ which, by \eqref{XjYjbound}, satisfy $|\zeta_j'| \ll (\log K)^2$. Applying an exponential bound (e.g.\ Hoeffding's inequality) to $\zeta_j'$ with a large enough constant $\rho >1$ leads to
\begin{multline*}
\Pr \bigg( \bigg| \sum_{j=1}^{R+1} (Y_j (x) - \mathbb{E} Y_j (x)) \bigg| \ge \rho R^{1/2} (\log K)^{5/2} \bigg) \\ \ll \Pr \left( \left| \sum_{j=1}^{R+1} (\zeta_j'-\mathbb{E} \zeta_j') \right| \ge \frac{\rho}{2} R^{1/2} (\log K)^{5/2} \right) + K^{-100} \ll K^{-100} .
\end{multline*}
The previous two formulas, together with the fact that $\mathbb{E}U(x)=O((\log K)^2)$ is negligible, give
\[ \sup_{C \subseteq \mathbb{C}} \left| \Pr \left( \frac{U(x)}{\sigma(\alpha) \sqrt{\log q_K}} \in C \right) - \Phi (C) \right| \ll \frac{K^{1/2} \log K}{R} + \frac{R^{1/2} (\log K)^{5/2}}{K^{1/2}} . \]
The optimal choice is $R= \lfloor K^{2/3}/\log K \rfloor$, in which case the error term is $K^{-1/6} (\log K)^2$. This finishes the proof of \eqref{TNCLT}.

\subsubsection{Completing the proof}\label{CLTproofsection}

\begin{proof}[Proof of Theorem \ref{2dimCLTtheorem}] Let $M \ge 3$ be an integer with Ostrowski expansion $M=\sum_{k=0}^{K-1} b_k q_k$, $b_{K-1} \neq 0$, and let $M^*=\sum_{k=K-k_0}^{K-1} b_k q_k$, where $k_0=\lfloor \frac{10 \log K}{\log (1/\alpha)} \rfloor$, as before. Consider a partition $[0,M^*) = \bigcup_{\ell=K-k_0}^{K-1} \bigcup_{b=0}^{b_\ell -1} [M_{\ell,b}, M_{\ell,b}+q_{\ell})$ with integers $0 \le M_{\ell, b} \ll q_{\ell}$. The central limit theorem \eqref{TNCLT} applied to $T_N$, $0 \le N <q_{\ell}$ defined with $\xi=M_{\ell,b}+1/2$ and $Q_{\ell}=M_{\ell,b}+q_{\ell}$ gives that for any convex set $C \subseteq \mathbb{C}$,
\[ \frac{1}{q_{\ell}} \left| \left\{ 0 \le N <q_{\ell} \, : \, \frac{T_N}{\sigma (\alpha) \sqrt{\log q_{\ell}}} \in C \right\} \right| = \Phi (C) + O \left( \frac{(\log K)^2}{K^{1/6}} \right) . \]
Setting $\Gamma_N = (\log P_N(\alpha), \pi S_N(\alpha))$ and $E_N=(\frac{1}{2} \log N, \pi E(\alpha) \log N)$, formula \eqref{identification} states that $\Gamma_{M_{\ell,b}+N}-E_{M_{\ell,b}+q_{\ell}}$ is identified with $i T_N$, $0 \le N <q_{\ell}$ up to a negligible error $O(\log K)$. Note that the standard Gaussian is invariant under the rotation corresponding to multiplication by $i$. Hence for any convex set $C \subseteq \mathbb{R}^2$,
\[ \left| \left\{ 0 \le N <q_{\ell} \, : \, \frac{\Gamma_{M_{\ell,b}+N}-E_{M_{\ell,b}+q_{\ell}}}{\sigma (\alpha) \sqrt{\log q_{\ell}}} \in C \right\} \right| = \Phi (C) q_{\ell} + O \left( q_{\ell} \frac{(\log K)^2}{K^{1/6}} \right) . \]
Here $E_{M_{\ell,b}+q_{\ell}}$ resp.\ $\sqrt{\log q_{\ell}}$ can be replaced by $E_{M_{\ell,b}+N}$ resp.\ $\sqrt{\log (M_{\ell,b}+N)}$ up to a negligible error $O(q_{\ell} (\log K)/K)$. We tacitly used the fact that $\Phi ((1+\varepsilon)C)=\Phi (C)+O(|\varepsilon |)$ uniformly in convex sets $C$. Hence summing over $\ell,b$ leads to
\[ \left| \left\{ 0 \le N < M^* \, : \, \frac{\Gamma_N-E_N}{\sigma(\alpha) \sqrt{\log N}} \in C \right\} \right| = \Phi(C) M^* + O \left( M^* \frac{(\log K)^2}{K^{1/6}} \right) . \]
Since $M-M^*\ll q_{K-k_0} \ll M/K^{10}$, we finally deduce
\[ \frac{1}{M} \left| \left\{ 0 \le N < M \, : \, \frac{\Gamma_N-E_N}{\sigma(\alpha) \sqrt{\log N}} \in C \right\} \right| = \Phi(C) + O \left( \frac{(\log \log M)^2}{(\log M)^{1/6}} \right) . \]
\end{proof}

\subsection{Non-badly approximable irrationals}\label{nonbadlysection}

Fix an arbitrary irrational $\alpha =[a_0;a_1,a_2, \dots ]$ with convergents $p_k/q_k=[a_0;a_1,\dots , a_k]$. Using a reciprocity formula of Bettin and Drappeau \cite{BD}, in our recent paper \cite[Equation (14) and Proposition 3]{AB} we observed that
\begin{equation}\label{maxpn}
\max_{0 \le N <q_k} \log P_N (\alpha ) = V (a_1+\cdots +a_k) +O(A_k+(\log A_k) k)
\end{equation}
with $V$ as in \eqref{V}, $A_k=\max_{1 \le \ell \le k} a_{\ell}$ and a universal implied constant. We also record a variant for future reference, see \cite[Equation (5) and Proposition 3]{AB}:
\begin{equation}\label{maxpn2}
\max_{0 \le N <q_k} \log P_N (\alpha ) =(V+o(1))(a_1+\cdots +a_k) +O \left( \frac{\log A_k}{a_{k+1}} \right) \quad \textrm{as } k \to \infty ,
\end{equation}
provided that $(a_1+\cdots +a_k)/k \to \infty$. We also proved \cite[Proposition 2]{AB} that for any $0 \le N <q_k$,
\begin{equation}\label{symmetric}
\log P_N (\alpha ) + \log P_{q_k-N-1} (\alpha ) = \log q_k + O \left( \frac{\log A_k}{a_{k+1}} \right)
\end{equation}
with a universal implied constant. In particular, the distribution of $\log P_N (\alpha )$ when $N$ is chosen randomly from $[0,q_k)$ is approximately symmetric about the center $\frac{1}{2} \log q_k$. Concerning the variance, the Diophantine sum in \eqref{variance} has the evaluation \cite{BO}
\begin{equation}\label{diophantinesumpartialquotients}
\sqrt{\sum_{1 \le m < q_k} \frac{1}{8 \pi^2 m^2 \| m \alpha \|^2}} = \frac{\pi}{\sqrt{720}} \sqrt{a_1^2+\cdots +a_k^2} +O(\sqrt{k})
\end{equation}
with a universal implied constant. Although formulas \eqref{maxpn}--\eqref{diophantinesumpartialquotients} hold for arbitrary irrationals, all but \eqref{symmetric} are useless for badly approximable ones; indeed, in that case the main terms and the error terms are of the same order of magnitude. Note also that \eqref{diophantinesumpartialquotients} gives a wrong value of $\sigma (\alpha )$ for quadratic irrationals. They give precise results, however, for certain non-badly approximable irrationals.

Consider first Euler's number $e$. Then \eqref{maxpn} and \eqref{diophantinesumpartialquotients} read
\[ \begin{split} \max_{0 \le N <q_k} \log P_N (e) &= \frac{V}{9} k^2 +O(k \log k) , \\ \sum_{1 \le m < q_k} \frac{1}{8 \pi^2 m^2 \| m e \|^2} &= \frac{\pi^2}{14580} k^3 +O(k^2) . \end{split} \]
Since $q_k$ increases at the rate $\prod_{\ell =1}^k a_{\ell} \le q_k \le \prod_{\ell =1}^k (a_{\ell} +1)$, we have $\log q_k = \frac{1}{3} k \log k +O(k)$. Relations \eqref{eexpected}, \eqref{evariance} and \eqref{eextremal} now follow easily from the previous three formulas and Proposition \ref{momentsprop}; for $\min_{1 \le N \le M} \log P_N (e)$ use also \eqref{symmetric}.
\begin{proof}[Proof of Theorem \ref{etheorem}] Using \eqref{evariance} and the Chebyshev inequality, the proof is entirely analogous to that of Theorem \ref{badlyapproxtheorem}.
\end{proof}

\subsection{Asymptotics almost everywhere}\label{aeasymptoticssection}

Let us recall the precise statistics of the partial quotients of almost every irrational. Fix a nondecreasing positive real-valued function $\varphi (x)$ on $(0,\infty )$. The two most fundamental properties are that for a.e.\ $\alpha$,
\begin{enumerate}
\item[(i)] the inequality $a_k \ge \varphi (k)$ has finitely many solutions $k \in \mathbb{N}$ if and only if $\sum_{k=1}^{\infty} 1/\varphi (k) <\infty$;

\item[(ii)] $\log q_k \sim \frac{\pi^2}{12 \log 2} k$ as $k \to \infty$.
\end{enumerate}
As for the sum of the partial quotients, a classical result of Khinchin \cite{KH} states that
\[ \frac{a_1+\cdots +a_k}{k \log k} \to \frac{1}{\log 2} \qquad \textrm{in measure.} \]
The same asymptotics only holds a.e.\ if the sum is ``trimmed'', i.e.\ the largest term is removed. More precisely, Diamond and Vaaler \cite{DV} proved that for a.e.\ $\alpha$,
\begin{equation}\label{DVtrimmed}
\frac{a_1+\cdots +a_k - \vartheta (\alpha, k) A_k}{k \log k} \to \frac{1}{\log 2} \quad \textrm{as } k \to \infty
\end{equation}
with some\footnote{Their proof actually gives that the same holds with $\vartheta (\alpha, k)=1$, but we will not need this fact.} $0 \le \vartheta (\alpha, k) \le 1$, where $A_k=\max_{1 \le \ell \le k} a_{\ell}$ as before. In particular, the arithmetic mean of the partial quotients diverges, and from \eqref{maxpn2} we immediately obtain that for a.e.\ $\alpha$,
\begin{equation}\label{aemaxpn}
\max_{0 \le N <q_k} \log P_N (\alpha ) = (V+o(1)) A_k +O(k \log k) \quad \textrm{as } k \to \infty
\end{equation}
with an implied constant depending on $\alpha$. By \eqref{symmetric} we also have
\begin{equation}\label{aeminpn}
\min_{0 \le N <q_k} \log P_N (\alpha ) = -(V+o(1)) A_k +O(k \log k) \quad \textrm{as } k \to \infty ,
\end{equation}
and formulas \eqref{aeupperbound}, \eqref{aelowerbound} follow.
\begin{proof}[Proof of Theorem \ref{aeupperdensitytheorem}] We shall need a further result of Diamond and Vaaler \cite[Lemma 2]{DV}. As a simple corollary of the mixing property of the partial quotients, they observed that given any $\varepsilon >0$, for a.e. $\alpha$ for all but finitely many $k \in \mathbb{N}$ there is at most one integer $1 \le \ell \le k$ with $a_{\ell} \ge k (\log k)^{1/2+\varepsilon}$.

Assume first, that $\varphi (k) \ge k \log k \log \log k$ for all $k$. By property (i) above, for a.e.\ $\alpha$ there are infinitely many $k_0 \in \mathbb{N}$ such that
\[ a_{k_0} \ge \varphi (k_0) \ge k_0 \log k_0 \log \log k_0 . \]
For all but finitely many of these $k_0$ we now choose an integer $k \in [k_0, 2k_0]$ as follows. First of all note that by the lemma of Diamond and Vaaler, for all $1 \le \ell \le 2k_0$, $\ell \neq k_0$ we have $a_{\ell} \ll k_0 (\log k_0)^{1/2+\varepsilon}$. In particular, $A_{2k_0}=\max_{1 \le \ell \le 2k_0} a_{\ell}=a_{k_0}$. Let us now partition $[k_0+1,2k_0]$ into $\sim k_0/(\log k_0)^2$ subintervals $I_1, I_2, \dots$, each of size $\sim (\log k_0)^2$. If $\max_{\ell \in I_j} a_{\ell} \ge (\log k_0)^4$ for all $j$, then clearly
\[ a_1+\cdots +a_{2k_0} \ge a_{k_0} + \sum_{j} (\log k_0)^4 \sim A_{2k_0} + k_0 (\log k_0)^2 , \]
which contradicts \eqref{DVtrimmed}. Therefore there exists a subinterval $I_j \subset [k_0+1,2k_0]$ of size $\sim (\log k_0)^2$ such that $\max_{\ell \in I_j} a_{\ell} \le (\log k_0)^4$. Choose $k$ in the middle of $I_j$. Note that the error term in Proposition \ref{momentsprop} then satisfies
\[ \max_{|\ell -k| \ll \log k} a_{\ell} \le \max_{\ell \in I_j} a_{\ell} \ll (\log k)^4 , \]
therefore
\[ \frac{1}{q_k} \sum_{N=0}^{q_k-1} \left( \log P_N (\alpha ) - \frac{1}{2} \log q_k \right)^2 = \sum_{m=1}^{q_k-1} \frac{1}{8 \pi^2 m^2 \| m \alpha \|^2} + O((\log k)^{12}) . \]
Using \eqref{DVtrimmed} once again, we obtain
\[ A_k \le \sqrt{a_1^2+\cdots +a_k^2} \le a_1+\cdots +a_k \le A_k+O(k \log k) , \]
hence \eqref{diophantinesumpartialquotients} yields
\[ \frac{1}{q_k} \sum_{N=0}^{q_k-1} (\log P_N (\alpha ))^2 = \frac{1}{q_k} \sum_{N=0}^{q_k-1} \left( \log P_N (\alpha ) - \frac{1}{2} \log q_k \right)^2 +O((\log q_k)^2) = \frac{\pi^2}{720} A_k^2 + O(A_k k \log k). \]
Since $A_k =A_{k_0} \gg k \log k \log \log k$, the previous formula together with  \eqref{aemaxpn} and \eqref{aeminpn} show that (cf.\ \eqref{aevariance})
\[ \sqrt{\frac{1}{q_k} \sum_{N=0}^{q_k-1} (\log P_N (\alpha ))^2} = \left( \frac{\pi}{\sqrt{720}V} +o(1) \right) \max_{0 \le N <q_k} |\log P_N (\alpha )| \qquad \textrm{for infinitely many }k. \]

Fix $\varepsilon >0$, and let $T_k=\max_{0 \le N <q_k} |\log P_N (\alpha )|$. Splitting the sum according to whether $|\log P_N (\alpha )| \le \varepsilon T_k$ or not, we get
\[ \frac{1}{q_k} \sum_{N=0}^{q_k-1} (\log P_N (\alpha ))^2 \le \varepsilon^2 T_k^2 + \frac{1}{q_k} \left| \left\{ 0 \le N < q_k \, : \, |\log P_N (\alpha )| \ge \varepsilon T_k \right\} \right| T_k^2 . \]
The previous two formulas immediately yield
\[ \frac{1}{q_k} \left| \left\{ 0 \le N < q_k \, : \, |\log P_N (\alpha )| \ge \varepsilon T_k \right\} \right| \ge \frac{\pi^2}{720 V^2} - \varepsilon^2 -o(1) \qquad \textrm{for infinitely many } k. \]
Recall from \eqref{symmetric} that the distribution of $\log P_N (\alpha )$ is approximately symmetric about the negligible center $\frac{1}{2} \log q_k \ll k=o(A_k)$. By construction we also have $k_0 \le k \le 2k_0$, and hence $0 \le N <q_k$ implies $\log N \le (\pi^2/(6 \log 2)+o(1)) k_0$. Therefore
\[ T_k = (V+o(1)) A_k \ge (V+o(1)) \varphi (k_0) \ge \frac{1}{100} \varphi \left( \frac{\log N}{100} \right) , \]
and we obtain
\[ \frac{1}{q_k} \left| \left\{ 0 \le N < q_k \, : \, \log P_N (\alpha ) \ge \frac{\varepsilon}{100} \varphi \left( \frac{\log N}{100} \right) \right\} \right| \ge \frac{\pi^2}{1440 V^2} - \frac{\varepsilon^2}{2} -o(1) \]
and
\[ \frac{1}{q_k} \left| \left\{ 0 \le N < q_k \, : \, \log P_N (\alpha ) \le - \frac{\varepsilon}{100} \varphi \left( \frac{\log N}{100} \right) \right\} \right| \ge \frac{\pi^2}{1440 V^2} - \frac{\varepsilon^2}{2} -o(1) \]
for infinitely many $k$. Note that $\varphi^*(x)=\frac{100}{\varepsilon} \varphi (100x)$ also satisfies $\sum_{k=1}^{\infty} 1/\varphi^* (k) =\infty$ and $\varphi^*(k) \ge k \log k \log \log k$. Hence by repeating the argument with $\varphi^*$, we get that for a.e.\ $\alpha$ the sets
\[ \left\{ N \in \mathbb{N} \, : \, \log P_N (\alpha ) \ge \varphi (\log N) \right\} \]
and
\[ \left\{ N \in \mathbb{N} \, : \, \log P_N (\alpha ) \le - \varphi (\log N) \right\} \]
both have upper asymptotic density at least $\pi^2/(1440 V^2) - \varepsilon^2/2$. Here $\varepsilon>0$ is arbitrary, and the claim follows under the initial assumption $\varphi (k) \ge k \log k \log \log k$. This assumption is easily removed: simply note that $\varphi^{**}(x)=\varphi (x)+x \log (x+3) \log \log (x+3)$ satisfies $\sum_{k=1}^{\infty} 1/\varphi^{**}(k) = \infty$ and $\varphi^{**}(k) \ge k \log k \log \log k$, and that the claim with $\varphi^{**}$ implies the claim with $\varphi$.
\end{proof}

\begin{proof}[Proof of Theorem \ref{aeliminftheorem}] Given $M \ge 1$, let $k_M(\alpha )$ be such that $q_{k_M(\alpha )} \le M < q_{k_M(\alpha )+1}$. From \eqref{maxpn2} and \eqref{DVtrimmed} we immediately deduce that for a.e.\ $\alpha$,
\[ \max_{1 \le N \le M} \log P_N (\alpha ) \ge \frac{V+o(1)}{\log 2} k_M(\alpha ) \log k_M(\alpha ) . \]
By $\log q_k \sim \frac{\pi^2}{12 \log 2}k$, here $k_M(\alpha ) \sim \frac{12 \log 2}{\pi^2} \log M$, and we get that for a.e.\ $\alpha$,
\[ \liminf_{M \to \infty} \frac{\displaystyle \max_{1 \le N \le M} \log P_N (\alpha )}{\log M \log \log M} \ge \frac{12 V}{\pi^2} . \]
Using \eqref{symmetric}, we similarly obtain that the same holds with $\max$ replaced by $-\min$. On the other hand, convergence in measure in Theorem \ref{convergenceinmeasuretheorem} shows that the a.e.\ limit along a suitably sparse subsequence equals $12V/\pi^2$.
\end{proof}

\subsection{Asymptotics in measure}\label{measuresection}

Recall that the Gauss measure $\nu (A):=\frac{1}{\log 2} \int_A \frac{1}{1+x} \, \mathrm{d}x$ ($A \subseteq [0,1]$ Borel) is an invariant measure under the shift $x \mapsto \{ 1/x \}$. The partial quotients are thus identically distributed:
\[ \nu \left( \left\{ \alpha \in [0,1] \, : \, a_k=j \right\} \right) = \frac{1}{\log 2} \log \left( 1+\frac{1}{j(j+2)} \right) , \quad k,j \ge 1. \]
The terms ``a.e.'' and ``convergence in measure'' are of course equivalent for the Gauss and the Lebesgue measures, and we will use them in this sense.

\begin{proof}[Proof of Theorem \ref{convergenceinmeasuretheorem}] Given $M \ge 1$, let $k_M (\alpha )$ be such that $q_{k_M(\alpha )} \le M < q_{k_M(\alpha )+1}$. By \eqref{maxpn2}, for a.e.\ $\alpha$,
\[ \max_{1 \le N \le M} \log P_N (\alpha ) = \left( V +o(1) \right) (a_1+\cdots +a_{k_M(\alpha )} + \xi_M(\alpha ) ) \]
with some $|\xi_M(\alpha )| \le a_{k_M(\alpha )+1}$. Formula \eqref{DVtrimmed} and $\log M \sim \frac{\pi^2}{12 \log 2} k_M(\alpha)$ show that for a.e.\ $\alpha$,
\[ \frac{\displaystyle \max_{1 \le N \le M} \log P_N (\alpha )}{\log M \log \log M} = \left( V +o(1) \right) \left( \frac{12}{\pi^2} +o(1)+\frac{\xi^*_M(\alpha)}{\log M \log \log M} \right) \]
with some $|\xi^*_M(\alpha) | \le 2 \max_{1 \le \ell \le k_M(\alpha )+1} a_{\ell}$. In particular, the $o(1)$ error terms converge to zero also in measure, and it remains to prove that
\begin{equation}\label{convinmeasure}
\frac{\max_{1 \le \ell \le k_M(\alpha )+1} a_{\ell}}{\log M \log \log M} \to 0 \qquad \textrm{in measure.}
\end{equation}
Since $\log M \sim \frac{\pi^2}{12 \log 2} k_M(\alpha)$ holds also in measure, we have
\[ \nu \left( \left\{ \alpha \in [0,1] \, : \, k_M(\alpha )+1 \ge 100 \log M \right\} \right) \to 0 \qquad \textrm{as } M \to \infty . \]
On the other hand, for any $\varepsilon >0$ and any $\ell \in \mathbb{N}$,
\[ \begin{split} \nu \left( \left\{ \alpha \in [0,1] \, : \, a_{\ell} \ge \varepsilon \log M \log \log M \right\} \right) &= \frac{1}{\log 2} \sum_{j \ge \varepsilon \log M \log \log M} \log \left( 1 +\frac{1}{j(j+2)} \right) \\ &\ll \frac{1}{\varepsilon \log M \log \log M} . \end{split} \]
The union bound yields
\[ \nu \left( \left\{ \alpha \in [0,1] \, : \, \max_{1 \le \ell \le 100 \log M} a_{\ell} \ge \varepsilon \log M \log \log M \right\} \right) \ll \frac{1}{\varepsilon \log \log M} \to 0 \quad \textrm{as } M \to \infty , \]
and \eqref{convinmeasure} follows. By \eqref{symmetric}, the claim also holds with $\max$ replaced by $-\min$.
\end{proof}

The proof of Theorem \ref{limitdistributiontheorem} is based on results of Samur \cite{SA} on the asymptotics of sums of the form $\sum_{\ell=1}^k f(a_{\ell})$. Let $\mu$ be a probability measure on $[0,1]$ that is absolutely continuous with respect to the Lebesgue measure. A special case of \cite[Corollary 2.13]{SA} states that for any $t \ge 0$,
\[ \mu \left( \left\{ \alpha \in [0,1] \, : \, \frac{a_1^2+\cdots +a_k^2}{k^2} \le t \right\} \right) \to F(t) \qquad \textrm{as } k \to \infty , \]
where Samur defines the limit distribution via the L\'evy--Khinchin representation of its characteristic function as
\[ \int_{\mathbb{R}} e^{itx} \, \mathrm{d} F(t) = \exp \left( i \frac{x}{\log 2} + \frac{1}{2\log 2} \int_0^{\infty} \left( e^{ixy}-1-ixy I_{(0,1]}(y) \right) y^{-3/2} \, \mathrm{d} y \right) . \]
Elementary calculations and the values of the Fresnel-type integrals
\[ \int_{0}^{\infty} \frac{1-\cos y}{y^{3/2}} \, \mathrm{d}y = \int_{0}^{\infty} \frac{\sin y}{y^{3/2}} \, \mathrm{d}y = \sqrt{2 \pi} \]
lead to the simplified form
\[ \int_{\mathbb{R}} e^{itx} \, \mathrm{d} F(t) = \exp \left( - \frac{\sqrt{2 \pi}}{2 \log 2} |x|^{1/2} \left( 1-i \, \mathrm{sgn} \, x \right) \right) , \]
which is the characteristic function of a L\'evy distribution supported on $[0,\infty )$. Using the appropriate normalization, in the special case of sums of squares of the partial quotients the theorem of Samur thus states that for any $t \ge 0$,
\begin{equation}\label{Samur}
\mu \left( \left\{ \alpha \in [0,1] \, : \, \frac{2 (\log 2)^2}{\pi} \cdot \frac{a_1^2+\cdots +a_k^2}{k^2} \le t \right\} \right) \to \int_0^t \frac{e^{-1/(2x)}}{\sqrt{2 \pi} x^{3/2}} \, \mathrm{d}x \qquad \textrm{as } k \to \infty .
\end{equation}
We shall also need a more sophisticated form of $\log q_k \sim \frac{\pi^2}{12 \log 2} k$, such as the central limit theorem \cite{PH,PS}
\begin{equation}\label{logqkCLT}
\mu \left( \left\{ \alpha \in [0,1] \, : \, \frac{\log q_k-\frac{\pi^2}{12 \log 2} k}{\tau \sqrt{k}} \le t \right\} \right) \to \int_{-\infty}^{t} \frac{1}{\sqrt{2 \pi}}e^{-x^2/2} \, \mathrm{d} x \quad \textrm{as } k \to \infty ,
\end{equation}
where $\tau >0$ is a universal constant.

\begin{proof}[Proof of Theorem \ref{limitdistributiontheorem}] Given $M \ge 1$, let $k_M (\alpha )$ be such that $q_{k_M(\alpha )} \le M < q_{k_M(\alpha )+1}$. By Proposition \ref{momentsprop}, for a.e.\ $\alpha$ there exists a constant $K(\alpha )>0$ such that
\[ \begin{split} \frac{1}{M} \sum_{N=1}^M \log P_N (\alpha ) &= \frac{1}{2} \log M + \zeta_M (\alpha ), \\ \frac{1}{M} \sum_{N=1}^M \left( \log P_N (\alpha ) - \frac{1}{2} \log M \right)^2 &= \sum_{m=1}^M \frac{1}{8 \pi^2 m^2 \| m \alpha \|^2} + \zeta_M^* (\alpha) \end{split} \]
with some
\[ \begin{split} |\zeta_M (\alpha )| &\le K(\alpha ) \max_{|\ell -k_M(\alpha ) | \le K(\alpha ) \log k_M(\alpha )} a_{\ell} \cdot \log \log M , \\ |\zeta_M^* (\alpha )| &\le K(\alpha ) \max_{|\ell -k_M(\alpha ) | \le K(\alpha ) \log k_M(\alpha )} a_{\ell}^2 \cdot (\log \log M)^4 . \end{split} \]
First, we estimate $\zeta_M(\alpha)$ in Gauss measure. Fix $\varepsilon >0$. Then $\nu (\{ \alpha \in [0,1] \, : \, K(\alpha ) \le K \})>1-\varepsilon$ with a large enough $K>0$. Letting $k_0:=\lceil \frac{12 \log 2}{\pi^2} \log M \rceil$, the central limit theorem \eqref{logqkCLT} immediately gives
\[ \nu \left( \left\{ \alpha \in [0,1] \, : \, |\log q_{k_0} - \log M| \le C \sqrt{\log M} \right\} \right) > 1-\varepsilon \]
with a large enough $C>0$ and all $M \ge M_0(\varepsilon)$. Since the convergent denominators increase at least exponentially fast (recall e.g.\ the general fact $q_{\ell+2}/q_{\ell} \ge 2$), for all $\alpha$ in the previous set we have $|k_M(\alpha ) - k_0 |\le C \sqrt{\log M}$. Therefore for all $\alpha$ outside a set of $\nu$-measure $< 2 \varepsilon$ and all large enough $M$,
\[ \{ \ell \, : \, |\ell -k_M (\alpha )| \le K(\alpha ) \log k_M(\alpha ) \} \subset [k_0-C \sqrt{\log M}, k_0+C \sqrt{\log M} ] , \]
and in particular
\[ |\zeta_M(\alpha )| \le K \max_{\ell \in [k_0-C \sqrt{\log M}, k_0+C \sqrt{\log M} ]} a_{\ell} \cdot \log \log M . \]
For any $\ell$,
\[ \nu \left( \left\{ \alpha \in [0,1] \, : \, a_{\ell} \ge \frac{C}{\varepsilon}\sqrt{\log M} \right\} \right) \ll \frac{\varepsilon}{C\sqrt{\log M}} , \]
therefore the union bound yields
\[ \nu \left( \left\{ \alpha \in [0,1] \, : \, \max_{\ell \in [k_0-C \sqrt{\log M}, k_0+C \sqrt{\log M} ]} a_{\ell} \ge \frac{C}{\varepsilon}\sqrt{\log M} \right\} \right) \ll \varepsilon . \]
Hence outside a set of $\nu$-measure $\ll \varepsilon$, we have $|\zeta_M(\alpha)| \le (CK/\varepsilon) \sqrt{\log M} \log \log M$. In particular, $\zeta_M(\alpha)=o(\sqrt{\log M} (\log \log M)^2)$ in measure, which proves the first claim of the theorem.

We similarly deduce $\zeta_M^*(\alpha )=o(\log M (\log \log M)^5)$ in measure. The same arguments also show that
\[ \begin{split} a_1^2+\cdots +a_{k_M(\alpha )}^2 &= a_1^2+\cdots + a_{k_0}^2 +O\left( \max_{|\ell -k_0| \le C \sqrt{\log M}} a_{\ell}^2 \cdot C \sqrt{\log M} \right) \\ &= a_1^2+\cdots +a_{k_0}^2 +o \left( (\log M)^{3/2} \log \log M \right) \quad \textrm{in measure} . \end{split} \]
Using e.g.\ \eqref{Samur}, we observe that
\[ \sqrt{(a_1^2+\cdots +a_{k_M (\alpha )}^2) k_M (\alpha)} =o \left( (\log M)^{3/2} \log \log M \right) \quad \textrm{in measure} \]
is also negligible, hence \eqref{diophantinesumpartialquotients} yields
\[ \frac{1}{M} \sum_{N=1}^M \left( \log P_N(\alpha ) - \frac{1}{2} \log M \right)^2 = \frac{\pi^2}{720} (a_1^2+\cdots +a_{k_0}^2) +o\left( (\log M)^{3/2} \log \log M \right) \]
in measure. In particular,
\[ \frac{10 \pi}{M (\log M)^2} \sum_{N=1}^M \left( \log P_N(\alpha ) - \frac{1}{2} \log M \right)^2 = \frac{2 (\log 2)^2}{\pi} \cdot \frac{a_1^2+\cdots +a_{k_0}^2}{k_0^2} +o(1) \quad \textrm{in measure}. \]
Since $\mu$ is absolutely continuous with respect to the Gauss measure, the $o(1)$ error term converges to zero also in $\mu$-measure. The second claim of the theorem is thus reduced to the result \eqref{Samur} of Samur.
\end{proof}

\section{Distribution of Birkhoff sums}\label{birkhoffdistribution}

Throughout this section, $f: \mathbb{R} \to \mathbb{R}$ is a $1$-periodic function which is of bounded variation on $[0,1]$, such that $\int_0^1 f(x) \, \mathrm{d} x =0$. Let $V(f)$ denote its total variation on $[0,1]$, and $\widehat{f}(m)=\int_0^1 f(x) e^{-2 \pi i m x} \, \mathrm{d} x$, $m \in \mathbb{Z}$ its Fourier coefficients. Let $\alpha =[a_0;a_1,a_2,\dots ]$ be irrational with convergents $p_k/q_k=[a_0;a_1,\dots, a_k]$. Finally, let $S_N(\alpha, f)=\sum_{n=1}^N f (n \alpha )$, and let $A_M=A_M(\alpha, f)$ and $B_M=B_M(\alpha, f)$ be as in \eqref{AMBM}.

\subsection{General functions of bounded variation}\label{generalsection}

\begin{prop}\label{momentsprop2} Assume that $a_k \le c k^d$ for all $k \ge 1$ with some constants $c,d \ge 0$. For any $q_k \le M < q_{k+1}$ we have
\begin{equation}\label{expectedSN}
\begin{split} A_M = &\sum_{\substack{m \in \mathbb{Z} \\ 0<|m|<M (\log M)^{2d+1}}} \left( 1-\frac{|m|}{\lfloor M (\log M)^{2d+1} \rfloor} \right) \widehat{f}(m) \frac{e^{2 \pi i m \alpha}}{1-e^{2 \pi i m \alpha}} \\ &+ O \left( V(f) \max_{|\ell -k| \ll \log k} a_{\ell} \cdot (\log \log M)^2 \right) \end{split}
\end{equation}
and
\begin{equation}\label{varianceSN}
B_M^2 = \sum_{m=1}^M \frac{|\widehat{f}(m)|^2}{2 \pi^2 \| m \alpha \|^2} + O \bigg( V(f) \sqrt{\sum_{m=1}^M \frac{|\widehat{f}(m)|^2}{\| m \alpha \|^2}} + V(f)^2 \max_{|\ell -k|\ll \log k} a_{\ell}^2 \cdot (\log \log M)^4 \bigg)
\end{equation}
with implied constants depending only on $c$ and $d$.
\end{prop}
\noindent The result applies also to orbits with an arbitrary starting point $x_0$, since $\sum_{n=1}^N f(x_0+n\alpha ) = \sum_{n=1}^N \tilde{f}(n \alpha)$ with the shifted function $\tilde{f}(x)=f(x+x_0)$. The variance $B_M^2$ does not depend on the starting point, as $|\widehat{f}(m)|$ is invariant under shifts.

Recall that $|\widehat{f}(m)|\le V(f)/|m|$ for all $m \neq 0$. In particular, by \eqref{m2malpha2sum} for all badly approximable $\alpha$ we have $B_M^2 \ll V(f)^2 \log M$ with an implied constant depending only on $\alpha$. If in addition $|\widehat{f}(m)| \ge C/|m|$ for all $m \neq 0$ with some constant $C>0$, then we have the matching lower bound $B_M^2 \gg C^2 \log M$. For more general $\alpha$,
\[ B_M^2 \ll V(f)^2 (a_1^2 + \cdots +a_k^2) \]
with a matching lower bound if $|\widehat{f}(m)| \ge C/|m|$, see \eqref{diophantinesumpartialquotients}.

Using Koksma's inequality \cite[p.\ 143]{KN} and the fact that the sequence $\{ n \alpha \}$, $1 \le n \le N$ has discrepancy $\ll (\log N)^{d+1}$ \cite[p.\ 52]{DT}, we get $|S_N(\alpha ,f)| \ll V(f) (\log N)^{d+1}$. In particular, $|A_M| \ll V(f) (\log M)^{d+1}$. If $f$ is even, then the expected value $A_M$ is zero up to the error in \eqref{expectedSN}. This follows from combining the $m$ and $-m$ terms in the main term of \eqref{expectedSN}, and noting that
\[ \left| \sum_{0<m<M (\log M)^{d+1}} \left( 1-\frac{m}{\lfloor M (\log M)^{d+1} \rfloor} \right) \widehat{f}(m) \right| \ll V(f) . \]
In fact, the series $\sum_{m=1}^{\infty} \widehat{f}(m)$ converges since $f$ is of bounded variation.

We mention that the first error term $V(f) \sqrt{\sum_{m=1}^M |\widehat{f}(m)|^2 / \| m \alpha \|^2}$ can be removed in \eqref{varianceSN} for certain $f$, such as $f(x)=\{ x \} -1/2$, or $f(x)=I_{[a,b]}(\{ x \}) -(b-a)$ provided that $\inf_{q \in \mathbb{Z} \backslash \{ 0 \}} |q| \cdot \| q \alpha -\beta \| >0$ for both endpoints $\beta =a,b$. This can be seen by following the steps in the proof of Proposition \ref{momentsprop}, and noting that the Fourier series of these functions satisfy a suitable tail estimate, such as
\[ \left| \sum_{\substack{m \in \mathbb{Z} \\ |m| > M (\log M)^{2d+2}}} \widehat{f}(m) e^{2 \pi i m x} \right| \ll \frac{1}{M (\log M)^{2d+2} \| x -a \|} + \frac{1}{M (\log M)^{2d+2} \| x -b \|} \]
in the case of $f(x)=I_{[a,b]}(\{ x \}) -(b-a)$. For general $f$ of bounded variation, we use an approximation formula for sums of the form $\sum_{n=1}^N f(x_n)$ instead: if the points $x_1, x_2, \dots, x_N \in \mathbb{R}$ satisfy $\| x_n-x_{\ell} \| \ge r>0$ for all $n \neq \ell$ with some $r>0$, then for any integer $H>1$,
\begin{equation}\label{fourierapproximation}
\sum_{n=1}^N f(x_n) = \sum_{\substack{m \in \mathbb{Z} \\ 0<|m| < H}} \left( 1-\frac{|m|}{H} \right) \widehat{f}(m) \sum_{n=1}^N e^{2 \pi i m x_n} + O \left( V(f) \left( \frac{\log H}{rH}+1 \right) \right)
\end{equation}
with a universal implied constant. The proof is based on smoothing with the Fej\'er kernel of order $H$ and Koksma's inequality, see \cite{BB}.

\begin{proof}[Proof of Proposition \ref{momentsprop2}] We may assume that $V(f)=1$. Fix $1 \le N \le M$, and note that the points $x_n=n \alpha$, $1 \le n \le N$ satisfy $\| n \alpha - \ell \alpha \| \ge \| q_k \alpha \| \ge 1/(2q_{k+1}) \gg 1/(M (\log M)^d)$ for all $n \neq \ell$ by the assumption $a_k \le ck^d$. Applying \eqref{fourierapproximation} with a suitable $r \gg 1/(M (\log M)^d)$ and $H=\lfloor M (\log M)^{2d+1} \rfloor$, we thus get
\[ S_N (\alpha, f) = \sum_{\substack{m \in \mathbb{Z} \\ 0<|m|<H}} \left( 1-\frac{|m|}{H} \right) \widehat{f}(m) \frac{e^{2 \pi i m \alpha}-e^{2 \pi i N m \alpha}}{1-e^{2 \pi i m \alpha}} +O(1). \]
Taking the average over $N \in [1,M]$ and introducing
\[ E_M := \sum_{\substack{m \in \mathbb{Z} \\ 0<|m|<H}} \left( 1-\frac{|m|}{H} \right) \widehat{f}(m) \frac{e^{2 \pi i m \alpha}}{1-e^{2 \pi i m \alpha}} , \]
we obtain
\[ \begin{split} A_M &= E_M - \sum_{\substack{m \in \mathbb{Z} \\ 0<|m|<H}} \left( 1-\frac{|m|}{H} \right) \widehat{f}(m) \frac{e^{2 \pi i m \alpha} - e^{2 \pi i m (M+1) \alpha}}{M(1-e^{2 \pi i m \alpha})^2} +O(1) \\ &= E_M + O \left( 1+ \sum_{m=1}^{H-1} \frac{1}{m} \min \left\{ \frac{1}{M \| m \alpha \|^2} , \frac{1}{\| m \alpha \|} \right\} \right) \\ &= E_M + O \left( \max_{|\ell -k| \ll \log k} a_{\ell} \cdot (\log \log M)^2  \right) , \end{split} \]
where the last step follows from an obvious modification of the second claim in Lemma \ref{diophantinelemma}. This proves \eqref{expectedSN}.

Next, we prove \eqref{varianceSN}. By the previous three formulas,
\[ S_N (\alpha ,f) - E_M = - \sum_{\substack{m \in \mathbb{Z} \\ 0<|m|<H}} \left( 1-\frac{|m|}{H} \right) \widehat{f}(m) \frac{e^{2 \pi i N m \alpha}}{1-e^{2 \pi i m \alpha}} + O(1) . \]
Squaring both sides, and then taking the average over $N \in [1,M]$ yields
\[ \frac{1}{M} \sum_{N=1}^M \left( S_N (\alpha ,f) -E_M \right)^2 = V_M +O \left( \sqrt{V_M} +1 \right) , \]
where
\[ V_M := \frac{1}{M} \sum_{N=1}^M \left| \sum_{\substack{m \in \mathbb{Z} \\ 0<|m|<H}} \left( 1-\frac{|m|}{H} \right) \widehat{f}(m) \frac{e^{2 \pi i N m \alpha}}{1-e^{2 \pi i m \alpha}} \right|^2 . \]
Following the steps in the proof of Proposition \ref{momentsprop} shows that the contribution of the off-diagonal terms is negligible, hence
\[ V_M = \sum_{\substack{m \in \mathbb{Z} \\ 0<|m|<H}} \left( 1-\frac{|m|}{H} \right)^2 \frac{|\widehat{f}(m)|^2}{|1-e^{2 \pi i m \alpha}|^2} + O \left( \max_{|\ell -k| \ll \log k} a_{\ell}^2 \cdot (\log \log M)^4 \right) . \]
We can clean up the main term by noting that the contribution of the terms $M<|m|<H$ is negligible, and that for all $0<|m| \le M$ the error of replacing $(1-|m|/H)^2$ by $1$ is also negligible:
\[ \begin{split} \sum_{\substack{m \in \mathbb{Z} \\ M<|m|<H}} \frac{1}{m^2 \| m \alpha \|^2} &\ll \max_{|\ell -k| \ll \log k} a_{\ell}^2 \cdot \log \log M , \\ \sum_{\substack{m \in \mathbb{Z} \\ 0<|m| \le M}} \frac{|m|}{H} \cdot \frac{1}{m^2 \| m \alpha \|^2} &\ll \frac{M}{H} (\log M)^{2d+1} \ll 1, \end{split} \]
see the fourth claim in Lemma \ref{diophantinelemma}. Therefore
\[ V_M = \sum_{m=1}^M \frac{|\widehat{f}(m)|^2}{2 \sin^2 (\pi m \alpha )} + O \left( \max_{|\ell -k| \ll \log k} a_{\ell}^2 \cdot (\log \log M)^4 \right) , \]
and using $1/\sin^2 (\pi m \alpha ) = 1/(\pi^2 \| m \alpha \|^2) +O(1)$ we deduce
\[ \frac{1}{M} \sum_{N=1}^M \left( S_N (\alpha ,f) -E_M \right)^2 = \sum_{m=1}^M \frac{|\widehat{f}(m)|^2}{2 \pi^2 \| m \alpha \|^2} + O \left( \sqrt{\sum_{m=1}^M \frac{|\widehat{f}(m)|^2}{\| m \alpha \|^2}} + \max_{|\ell -k| \ll \log k} a_{\ell}^2 \cdot (\log \log M)^4 \right) . \]
Finally, the first claim \eqref{expectedSN} shows that the error of replacing $E_M$ by $A_M$ on the left-hand side is negligible, and \eqref{varianceSN} follows.
\end{proof}

\subsection{Indicators of intervals}

In the special case $f(x)=I_{[a,b]}(\{ x \}) - (b-a)$ with some $[a,b] \subset [0,1]$ and a badly approximable $\alpha$, Proposition \ref{momentsprop2} gives
\[ B_M^2 = \sum_{m=1}^M \frac{\sin^2 (\pi m (b-a))}{2 \pi^4 m^2 \| m \alpha \|^2} + O \left( \sqrt{\log M} \right) \]
with an implied constant depending only on $\alpha$. The order of magnitude of $B_M^2$ thus depends on the Diophantine approximation properties of the vector $(\alpha , b-a)$. As we observed in the previous section, $B_M^2 \ll \log M$ holds for any $[a,b] \subset [0,1]$. To establish a lower bound, we need to find positive integers $m$ for which $\| m \alpha \|$ is small, but $\| m(b-a) \|$ is bounded away from zero. In other words, we want the sequence $(m \alpha, m(b-a)) \pmod{1}$ to visit a narrow but long rectangle in $\mathbb{R}^2 / \mathbb{Z}^2$ many times. Our proof of Theorem \ref{BUtheorem} is based on a quantitative form of Kronecker's theorem due to Cassels \cite[p.\ 97--99]{CA}.
\begin{lem}[Cassels]\label{casselslemma} Let $\alpha_1, \dots, \alpha_n, \beta_1, \dots, \beta_n \in \mathbb{R}$, and let $C_1, \dots, C_n,Q>0$ be reals. Assume that for any integral vector $(m_1, \dots, m_n) \in \mathbb{Z}^n$,
\[ \| m_1 \beta_1 + \cdots +m_n \beta_n \| \le \frac{2^n}{((n+1)!)^2} \max \left\{ C_1 |m_1|, \dots, C_n |m_n|, Q \| m_1 \alpha_1 + \cdots +m_n \alpha_n \| \right\} . \]
Then the system of $n+1$ inequalities
\[ \begin{split} \| q \alpha_k - \beta_k \| &\le C_k \quad (1 \le k \le n), \\ |q| &\le Q \end{split} \]
has an integral solution $q \in \mathbb{Z}$.
\end{lem}

\begin{proof}[Proof of Theorem \ref{BUtheorem}] We give separate proofs for rational and irrational values of $b-a$. First, assume that $b-a$ is irrational. Fix a constant $K>0$ such that
\[ \| q \alpha \| \ge \frac{K}{|q|}, \qquad \| k(b-a) \| \ge K, \qquad \| q \alpha -k(b-a) \| \ge K/|q| \]
for all integers $q \neq 0$ and all $k \in \{ 1,3,5,7,9 \}$. Let us apply Lemma \ref{casselslemma} with $n=2$ and
\[ \alpha_1=\alpha , \qquad \alpha_2 =b-a, \qquad \beta_1 =0, \qquad \beta_2 =\frac{1}{2}, \qquad C_1 = \frac{100}{KQ}, \qquad C_2 = 0.49, \]
and a large real number $Q>5/K$. We need to check that for all $(m_1, m_2) \in \mathbb{Z}^2$,
\begin{equation}\label{casselsinequality}
\| m_2 / 2 \| \le \frac{1}{9} \max \left\{ \frac{100|m_1|}{KQ}, 0.49 |m_2|, Q \| m_1 \alpha + m_2 (b-a) \| \right\} .
\end{equation}
This trivially holds if $m_2$ is even, and also if $|m_2| \ge 11$. We may thus assume that $|m_2| \in \{ 1,3,5,7,9 \}$, in which case the left hand side is $1/2$. If $m_1=0$, then the last term in the maximum satisfies $Q \| m_2 (b-a) \| \ge Q K > 5$, and \eqref{casselsinequality} follows. If $m_1 \neq 0$, then
\[ \max \left\{ \frac{100|m_1|}{KQ} , Q \| m_1 \alpha + m_2 (b-a) \| \right\} \ge \max \left\{ \frac{100 |m_1|}{KQ}, \frac{KQ}{|m_1|} \right\} \ge 10, \]
and \eqref{casselsinequality} follows once again.

Hence for any $Q>5/K$ there exists an integer $|q| \le Q$ with $\| q \alpha \| \le 100/(KQ)$ and $\| q (b-a) -1/2 \| \le 0.49$; the latter inequality is equivalent to $\| q (b-a) \| \ge 1/100$. Since $\| q \alpha \| \ge K/|q|$, the integer solution $q$ also satisfies $|q| \ge (K^2/100)Q$. Note that $-q$ is a solution if and only if $q$ is. In particular, there exist $\gg \log M$ integers $m \in [1,M]$ such that $\| m \alpha \| \ll 1/m$ and $\| m (b-a) \| \ge 1/100$. Therefore
\[ \sum_{m=1}^M \frac{\sin^2 (\pi m (b-a))}{2 \pi^4 m^2 \| m \alpha \|^2} \gg \log M, \]
and the claim follows.

Next, assume that $0<b-a<1$ is rational, say $b-a=s/r$ with relatively prime positive integers $s,r$. Then $\sin^2 (\pi m (b-a)) \gg 1$ for any integer $m$ not divisible by $r$. Applying Lemma \ref{casselslemma} with $n=1$; $\alpha_1=r\alpha$; $\beta_1=-\alpha$ and a large real $Q>0$, we similarly deduce that there exists an integer $Q \ll |q| \le Q$ such that $\| (qr+1) \alpha \| \ll 1/Q$. In particular, there exist $\gg \log M$ integers $m \in [1,M]$ such that $\| m \alpha \| \ll 1/m$ and $m \equiv \pm 1 \pmod{r}$, and the claim follows.
\end{proof}

Finally, we give an alternative proof of Theorem \ref{BUtheorem} for rational $b-a=s/r$. As observed, it is enough to show that there exist $\gg \log M$ integers $m \in [1,M]$ such that $\| m \alpha \| \ll 1/m$, and $m$ is not divisible by $r$. The convergents to $\alpha$ satisfy the identity $q_{\ell} p_{\ell-1} - q_{\ell-1} p_{\ell} = (-1)^{\ell}$. Therefore among any two consecutive convergent denominators, at least one is not divisible by $r$. This yields $\gg \log M$ integers $m$ with the desired properties.

\subsection{Central limit theorem for functions of bounded variation}

\begin{proof}[Proof of Theorem \ref{BVCLTtheorem}] This is very similar to the proof of Theorem \ref{2dimCLTtheorem}, so we only indicate the necessary changes. Throughout, constants and implied constants depend only on $\alpha$.

We may assume that $V(f)=1$, and that $B_M \gg (\log M)^{1/3} (\log \log M)^2$. We saw in the proof of Proposition \ref{momentsprop2} (cf.\ \eqref{identification}) that, for $1 \le N \le M$,
\begin{equation}\label{SN-AM}
S_N(\alpha, f) - A_M = - \sum_{\substack{m \in \mathbb{Z} \\ 0<|m|< M \log M}} \left( 1-\frac{|m|}{\lfloor M \log M \rfloor} \right) \widehat{f}(m) \frac{e^{2 \pi i N m \alpha}}{1-e^{2 \pi i m \alpha}} + O((\log \log M)^2) .
\end{equation}
We will prove that given any $\xi \in \mathbb{R}$, with $M=q_K$ a convergent denominator and any $q_K \ll Q_K \ll q_K$,
\[ T_N(f) := \sum_{\substack{m \in \mathbb{Z} \\ 0<|m|< Q_K \log Q_K}} \left( 1-\frac{|m|}{\lfloor Q_K \log Q_K \rfloor} \right) \widehat{f}(m) \frac{e^{2 \pi i (N+\xi) m \alpha}}{1-e^{2 \pi i m \alpha}}, \quad 0 \le N <q_K \]
satisfies the central limit theorem: for any $t \in \mathbb{R}$,
\begin{equation}\label{TNfCLT}
\frac{1}{q_K} \left| \left\{ 0 \le N<q_K \, : \, \frac{T_N(f)}{B_{q_K}} \le t \right\} \right| = \int_{-\infty}^t \frac{e^{-y^2/2}}{\sqrt{2 \pi}} \, \mathrm{d} y + O \left( \frac{K^{1/3} (\log K)^2}{B_{q_K}} \right) .
\end{equation}
The main difference compared to the proof of Theorem \ref{2dimCLTtheorem} is that $T_N(f)$ is real-valued, as can be seen by combining the $m$ and $-m$ terms, thus the limit distribution is now a $1$-dimensional standard Gaussian.

Observe that the coefficients in $T_N(f)$ satisfy the same decay as those in $T_N$:
\[ \left| \left( 1-\frac{|m|}{\lfloor Q_K \log Q_K \rfloor} \right) \widehat{f}(m) \frac{1}{1-e^{2 \pi i m \alpha}} \right| \ll \frac{1}{|m| \cdot \| m \alpha \|} . \]
Hence if $N$ is chosen from $[0,q_K)$ uniformly at random, then $\mathbb{E}T_N(f)=O((\log K)^2)$ and $\mathbb{E} |T_N(f)|^2=B_{q_K}^2 + O(B_{q_K} (\log K)^2)$, see $V_M$ in the proof of Proposition \ref{momentsprop2}; an analogue of Lemma \ref{covariancelemma}. Let
\[ k_0 = \left\lfloor \frac{10 \log K}{\log (1/\alpha)} \right\rfloor \quad \textrm{and} \quad S= \left\{ 0<|m|< Q_K \log Q_K \, : \, H(|m|)-h(|m|) \le k_0 \right\} . \]
As an analogue of Lemma \ref{L2lemma2}, we deduce that
\[ T_N(f) = \sum_{m \in S} \left( 1-\frac{|m|}{\lfloor Q_K \log Q_K \rfloor} \right) \widehat{f}(m) \frac{e^{2 \pi i (g(N)+\xi) m \alpha}}{1-e^{2 \pi i m \alpha}} + Z_N (f), \quad 0 \le N<q_K \]
with some $Z_N(f)$, $0 \le N<q_K$ that satisfies $q_K^{-1} \sum_{N=0}^{q_K-1} |Z_N(f)|^2 \ll (\log K)^4$.

Let us choose $x$ from the interval $[0,\alpha^{-K})$ uniformly at random, and define
\[ U(x,f) = \sum_{m \in S} \left( 1-\frac{|m|}{\lfloor Q_K \log Q_K \rfloor} \right) \widehat{f}(m) \frac{e^{2 \pi i (G(x,m) + \xi m \alpha)}}{1-e^{2 \pi i m \alpha}} , \]
where $G(x,m)$, $m>0$ is as in Section \ref{markovchainsection}, extended via $G(x,-m):=-G(x,m)$ for negative integers. Similarly to Lemma \ref{UxCLTlemma}, we reduce the central limit theorem for $T_N(f)$ to that for $U(x,f)$: we have $\mathbb{E}U(x,f) = O((\log K)^2)$ and $\mathbb{E} |U(x,f)|^2 = B_{q_K}^2 + O(B_{q_K} (\log K)^2)$. Further,
\[ \sup_{t \in \mathbb{R}} \left| \Pr \left( \frac{T_N(f)}{B_{q_K}} \le t \right) - \Phi(t) \right| = \sup_{t \in \mathbb{R}} \left| \Pr \left( \frac{U(x,f)}{B_{q_K}} \le t \right) - \Phi(t) \right| + O \left( \frac{(\log K)^{4/3}}{B_{q_K}^{2/3}} \right) , \]
where $\Phi (t)=\int_{-\infty}^t e^{-y^2/2}/\sqrt{2 \pi} \, \mathrm{d} y$.

Following the steps in Section \ref{independentsection}, we can approximate $U(x,f)$ by a sum of independent random variables, leading to
\[ \sup_{t \in \mathbb{R}} \left| \Pr \left( \frac{U(x,f)}{B_{q_K}} \le t \right) - \Phi(t) \right| \ll \frac{K \log K}{R B_{q_K}} + \frac{R^{1/2} (\log K)^{5/2}}{B_{q_K}} . \]
The optimal choice is once again $R=\lfloor K^{2/3}/\log K \rfloor$, in which case the error term is $O(K^{1/3} (\log K)^2/B_{q_K})$. This establishes \eqref{TNfCLT}, and in particular Theorem \ref{BVCLTtheorem} for $M=q_K$ a convergent denominator.

Taking the average in \eqref{SN-AM} over $N \in [1,M']$ and applying Lemma \ref{diophantinelemma2}, we deduce that for any $3 \le M' \le M$,
\[ \begin{split} |A_M - A_{M'}| &= \left| \frac{1}{M'} \sum_{N=1}^{M'} \sum_{\substack{m \in \mathbb{Z} \\ 0<|m|< M \log M}} \left( 1-\frac{|m|}{\lfloor M \log M \rfloor} \right) \widehat{f}(m) \frac{e^{2 \pi i N m \alpha}}{1-e^{2 \pi i m \alpha}} \right| + O((\log \log M)^2) \\ &\ll \sum_{1 \le m < M \log M} \frac{1}{m \| m \alpha \|} \min \left\{ \frac{1}{M' \| m \alpha \|} ,1 \right\} + (\log \log M)^2 \\ &\ll \left( 1+\log \frac{M}{M'} \right)^2 + (\log \log M)^2 . \end{split} \]
By \eqref{varianceSN} and the fourth claim in Lemma \ref{diophantinelemma}, for any $3 \le M' \le M$,
\[ | B_{M}^2 - B_{M'}^2 | = \sum_{m=M'+1}^M \frac{|\widehat{f}(m)|^2}{2 \pi^2 \| m \alpha \|^2} + O(B_M) \ll \sum_{m=M'+1}^M \frac{1}{m^2 \| m \alpha \|^2} + B_M \ll \log \frac{M}{M'} + B_M . \]
We follow the steps in Section \ref{CLTproofsection} to extend the result from $M=q_K$ (a convergent denominator) to general $M$. With the notation of that section, formula \eqref{TNfCLT} gives
\[ \frac{1}{q_{\ell}} \left| \left\{ 0 \le N <q_{\ell} \, : \, \frac{S_{M_{\ell,b}+N} (\alpha, f) - A_{M_{\ell, b} +q_{\ell}}}{B_{q_{\ell}}} \le t \right\} \right| = \Phi (t) + O \left( \frac{(\log M)^{1/3} (\log \log M)^2}{B_{q_{\ell}}} \right) . \]
Here
\[ |A_{M_{\ell,b}+q_{\ell}}-A_M| \ll \left( 1 + \log \frac{q_K}{q_{K-k_0}} \right)^2 + (\log \log M)^2 \ll (\log \log M)^2 , \]
and
\[ \frac{B_{q_{\ell}}}{B_M} = 1 + O \left( \frac{\log \frac{q_K}{q_{K-k_0}}}{B_M^2} + \frac{1}{B_M} \right) = 1 + O \left( \frac{1}{B_M} \right) , \]
hence $A_{M_{\ell,b}+q_{\ell}}$ resp.\ $B_{q_{\ell}}$ can be replaced by $A_M$ resp.\ $B_M$ up to a negligible error $O((\log \log M)^2/B_M)$. Summing over $\ell,b$ leads to
\[ \frac{1}{M^*} \left| \left\{ 0 \le N<M^* \, : \, \frac{S_N (\alpha, f) - A_M}{B_M} \le t  \right\} \right| = \Phi (t) + O \left( \frac{(\log M)^{1/3}(\log \log M)^2}{B_M} \right) . \]
Here $M-M^* \ll M/(\log M)^{10}$, therefore we can replace $M^*$ by $M$. This finishes the proof of Theorem \ref{BVCLTtheorem} for general $M \ge 3$.
\end{proof}

\section*{Acknowledgement}

The author is supported by the Austrian Science Fund (FWF) projects Y 901 and F 5510, which is part of the Special Research Program (SFB) ``Quasi-Monte Carlo Methods: Theory and Applications''. I would like to thank the referee for a careful reading of the manuscript and valuable comments.

\end{document}